\documentclass[twoside]{IEEEtran}
\usepackage{amsmath,mathtools,amssymb, txfonts}
\usepackage{mathrsfs}
\usepackage{graphicx,epstopdf}
\usepackage{verbatim}
\usepackage{ifthen}
\usepackage{multirow}
\usepackage[normalem]{ulem}
\usepackage{cite}
\usepackage[usenames,dvipsnames]{color}
\usepackage{xcolor}
\usepackage{lscape}
\usepackage{mdwlist}
\usepackage{url}
\usepackage{dcolumn}
\usepackage{amsfonts}
\usepackage{latexsym}
\usepackage{bm}
\usepackage{soul}

\usepackage{url}
\usepackage{subcaption}
\usepackage{xcolor}
\usepackage{framed} 
\usepackage[framed]{ntheorem}
\usepackage{threeparttable}

\usepackage{algorithm}
\usepackage{algorithmicx}
\usepackage{algpseudocode}
\usepackage[color=yellow]{todonotes}
\usepackage{lipsum}
\usepackage{tikz}
\usepackage{tkz-graph}
\usepackage{float}

\DeclareMathOperator{\diag}{diag}

\newboolean{showcomments}
\setboolean{showcomments}{true}

\def\ba{\begin{array}}
	\def\ea{\end{array}}
\newcommand{\beq}{\begin{align}}
	\newcommand{\eeq}{\end{align}}
\newcommand{\bq}{\begin{eqnarray}}
	\newcommand{\eq}{\end{eqnarray}}
\newcommand{\bqn}{\begin{eqnarray*}}
	\newcommand{\eqn}{\end{eqnarray*}}
\newcommand{\bee}{\begin{enumerate}}
	\newcommand{\eee}{\end{enumerate}}
\newcommand{\bi}{\begin{itemize}}
	\newcommand{\ei}{\end{itemize}}
\newcommand{\btab}{\begin{tabular}}
	\newcommand{\etab}{\end{tabular}}

\newtheorem{theorem}{Theorem}
\newtheorem{definition}{Definition}
\newtheorem{lemma}{Lemma}
\newtheorem{corollary}{Corollary}

\newtheorem{remark}{Remark}
\newtheorem{assumption}{Assumption}

\newcommand{\hL}{\mathcal{L}}
\newcommand{\hN}{\mathcal{N}}

\newcommand{\cN}{{\cal N}}

\newcommand{\BEAS}{\begin{eqnarray*}}
\newcommand{\EEAS}{\end{eqnarray*}}
\newcommand{\BEA}{\begin{eqnarray}}
\newcommand{\EEA}{\end{eqnarray}}
\newcommand{\BEQ}{\begin{equation}}
\newcommand{\EEQ}{\end{equation}}
\newcommand{\BIT}{\begin{itemize}}
\newcommand{\EIT}{\end{itemize}}
\newcommand{\BNUM}{\begin{enumerate}}
\newcommand{\ENUM}{\end{enumerate}}

\begin{document}
	
	\title{\LARGE{Signal-Anticipat{ion} in Local Voltage Control in Distribution Systems}}
	
	\author{Zhiyuan Liu, Seungil You, Xinyang Zhou, Guohui Ding and Lijun Chen
	 \thanks{Z. Liu, G. Ding and L. Chen are with the College of Engineering and Applied Science, University of Colorado, Boulder, CO 80309, USA (emails: \{zhiyuan.liu, guohui.ding, lijun.chen\}@colorado.edu).}
		\thanks{S. You is with Kakao Mobility (email: euclid85@gmail.com).}
		\thanks{X. Zhou is with National Renewable Energy Laboratory, Golden, CO 80401, USA (email: xinyang.zhou@nrel.gov).}
		\thanks{Preliminary result of this paper has been presented at IEEE International Conference on Smart Grid Communications, Venice, Italy, 2014 \cite{Chen-2014-SA}.}
		\vspace{-2mm}
}
	\maketitle
	\pagestyle{plain}
	
	\begin{abstract}
    	
    	We consider the signal-anticipating behavior in local Volt/Var control for distribution systems. {Such a behavior makes interaction among the nodes a game.} We characterize Nash equilibrium of the game as the optimum of a global optimization problem and establish its asymptotic global stability. { We also show that the signal-anticipating voltage control has less restrictive convergence condition than the signal-taking control.} We then introduce the notion of Price of Signal-Anticipat{ion} (PoSA) to characterize the impact of signal-anticipating control, and use the gap in cost between network equilibrium in the signal-taking control and Nash equilibrium in the signal-anticipating control as the metric for PoSA. We characterize how the PoSA scales with the size, topology, and heterogeneity of the {distribution} network for a few network settings. 
Our results show that the PoSA is upper bounded by a constant and the average PoSA per node will go to zero as the size of the network {increases.} This is desirable as it means that the PoSA will not be arbitrarily large, no matter what the size of the network is, and no mechanism is needed to mitigate the signal-anticipating behavior. {We further carry out numerical experiments with a real-world distribution circuit to complement the theoretical analysis.}
 \end{abstract}
    


	\section{Introduction}\label{sect:intro}
	
	Motivated by the proposed IEEE 1547.8 standard \cite{New1547}, we have studied in \cite{Farivar13,zhou2018reverse} a class of inverter-based local Volt/Var control schemes in distribution systems that set the reactive power at the output of an inverter based only on the local voltage deviation from  its nominal value at a node/bus. 
	 We have shown that the resulting  dynamical system with such local voltage control  
	has a unique equilibrium point, characterized it as the unique 
	optimum of a well-defined global optimization problem, and established its global stability. See Section \ref{sect:model} for a brief review of the related results.  
	

	These Volt/Var control schemes  take the feedback signal, i.e., the voltage, as given, and therefore the nodes/buses can be seen as being {\em signal-taking}. However,  a node 
	may be able to learn or infer the impact of its own decision on the feedback signal, and take it into
	consideration when making the control decision on reactive power, which we call the {\em signal-anticipating} voltage control. In this paper, we study such a 
	signal-anticipating behavior. Specifically, the signal-anticipating control makes the interaction
	among the nodes a game, which we call the voltage control game. We show that the signal-anticipating voltage control is the best response algorithm of the voltage control game, 
	and its fixed point is the Nash equilibrium of the voltage control game and vice versa. We further show that the voltage control game has a unique 
	Nash equilibrium, characterize it as the optimum of a global optimization problem, and establish its asymptotic global stability under the signal-anticipating voltage control. {We also show that the signal-anticipating voltage control has less restrictive convergence condition than the signal-taking control.}

	The signal-taking and signal-anticipating behaviors in local voltage control are analogous to the price-taking and price-anticipating behaviors in economics \cite{MWG,nisan2007algorithmic}. It is well-known that in a competitive market with price-taking customers the system achieves an efficient equilibrium and in an oligopolistic market with price-anticipating customers the system usually incurs efficiency loss. Similarly, we consider that the signal-taking behavior leads to an efficient equilibrium while the signal-anticipating behavior may result in efficiency loss.
	 Specifically, we introduce the notion of the {\em price of signal-anticipat{ion}} (PoSA) to characterize the impact of the signal-anticipat{ion} in local voltage control in particular and in distributed control in general. We use the gap in cost between the network equilibrium in the signal-taking voltage control and the Nash equilibrium in the signal-anticipating voltage control as the metric for PoSA, and  characterize how it scales with the size, topology, and heterogeneity of the {distribution} network. 
	 
	 In particular, we give an upper bound on the PoSA, and show that it is bounded by a constant that is independent of the size of the network. We also show numerically that the upper bound is very tight for the linear networks and randomly generated tree networks when the size of the network is large. Further, for the linear networks, we give a refined characterization of the upper bound. This refinement is based on an analytical form for the inverse of the reactance matrix (or voltage to reactive power sensitivity matrix) we discover.\footnote{{Although it is not directly relevant to the goals of this work,  the discovery of the analytical form for the inverse of the reactance matrix is itself a notable contribution; see the Appendix. In \cite{zhou2018reverse} we use it  to understand the structural properties of the local voltage control in distribution systems.}} {We then carry out numerical experiments with a 42-bus distribution feeder of Southern California Edison (SCE) to examine the efficiency loss as well as the convergence of the signal-anticipating voltage control. All the numerical results confirm the afore-mentioned theoretical analysis.}

%
	
	To our knowledge, this paper is the first to study the signal-anticipating behavior and introduce the notion of the price of signal-anticipat{ion}. {The signal-anticipating behavior arises in the setting of distributed control, where network nodes or users have only limited and partial information about the network due to various practical constraints and expect to make the best use of the information in order to, e.g., improve the stability of the system or optimize individual objectives. For example, one motivation for the signal-anticipating local voltage control is that adapting to the expected voltage rather than the current voltage may result in better convergence property, as confirmed by Theorem~\ref{theorem:5}. Besides this engineering perspective, there is a complementary economic perspective: The signal-anticipating behavior arises from the network nodes or users being self-interested and strategic, the same way as the price-anticipating behavior arising in an oligopolistic market. In this paper, we take more of an engineering perspective, even though we use an individual objective function to guide a node in taking into consideration the impact of its own decision (see Equation \eqref{eq:sa}). } 
	
	{An understanding and characterization of the signal-anticipating behavior} will be insightful in designing mechanisms to mitigate its impact if it is not desired. Regarding the signal-anticipating behavior in local voltage control in distribution systems, our results show that the PoSA is upper bounded by a constant, independent of the size of the network, and the ``average'' PoSA per node will go to zero as the size of the network {increases}. This is desirable as it means that the PoSA will not be arbitrarily large, no matter what the size of the network is, and no mechanism is needed to mitigate the signal-anticipating behavior. 
	
	
	{\em Related work}: Voltage control is a research area with a huge literature. Traditional approach to voltage control in distribution systems is via capacitor banks and under load tap changers; see, e.g., \cite{Baran1989a, Baran1989b, ChiangBaran1990}. 
	The new inverter-based approach that can control reactive power much faster and in a much finer granularity has been proposed and studied in, e.g., 
		\cite{Turitsyn11, Smith2011, simpson2017voltage, zhu2016fast} with local control,	\cite{vsulc2014optimal, zhou2017incentive, zhou2017stochastic, li2014real, zhou2018hierarchical} with distributed control, and \cite{Farivar2011-VAR-SGC, dall2014optimal, kekatos2015stochastic} with centralized control.
	 The local voltage control based on real-time voltage measurement has also been proposed for transmission systems; see, e.g., \cite{Ilic1988}. 
    	
	The price-anticipating behavior has been studied for engineering systems too; see, e.g.,  \cite{feldman2005price,johari2006scalable}  for network resource allocation and \cite{chen2010two,samadi2012advanced,li2015demand,fan2012distributed} for demand management in smart grids. More generally, efficiency loss arising from strategic or self-interested behaviors has ben studied using different metrics such as the price of anarchy \cite{nisan2007algorithmic} for various systems including power systems and communication networks; see, e.g., \cite{johari2004efficiency,wu2006price,haviv2007price, chen2009impact,johari2011parameterized,altman2011load,chen2015interaction,li2015demand}, to just name a few. 
	
	{
	A comparison with our preliminary work \cite{Chen-2014-SA} is in place.  This paper has significantly extended and improved upon \cite{Chen-2014-SA} as follows.  First, the current proof of convergence of the signal-anticipating voltage control uses the Contraction Mapping Theorem instead of the Lyapunov Stability Theorem, and can handle a more general problem setting with non-smooth cost functions.  
Further, we have also showed that the convergence condition for the signal-anticipating voltage control is less restrictive than that of the  signal-taking control. 
Second, in  \cite{Chen-2014-SA} the analytical characterization of PoSA is very simple and for the extreme cases with very large provisional cost or very large voltage deviation cost, and the numerical examples are for the simple two-link network and linear network with all power lines having the same reactance. In contrast, our current analytical characterization of PoSA is much more complicated and thorough:  1) derive an upper bound on PoSA for arbitrary tree networks and investigate its tightness in view of a lower bound, and 2) derive a refined and closed form for the upper bound for general linear networks with arbitrary line reactances. Third, in addition to many numerical examples that are used to complement the analytical results, in this paper we have also carried out numerical experiments with a real-world distribution feeder. 
	}

	\begin{table}[htbp]
		\begin{small}
			\caption{Notations}
			\label{table: symbol}
			\begin{center}
				\begin{tabular}{ll}
					\hline
					$\hN$ & \!\!\!\!Set of nodes, excluding node 0, indexed as \{1,2,$\cdots$, \!n\}\\
					$\hL$ & \!\!\!\!Set of all links representing power lines\\
					$\hL_i$ & \!\!\!\!Set of links on the path from node 0 to node i \\ 
                    $p_i^c, p_i^g$ &\!\!\!\!Real power consumption and generation at node $i$\\
					$q_i^c, q_i^g$& \!\!\!\!Reactive power consumption and generation at node $i$\\
					$r_{ij}, x_{ij}$ & \!\!\!\!Resistance and reactance of line $(i,j)$\\
					$P_{ij}, Q_{ij}$ & \!\!\!\!Real and reactive power flows from $i$ to $j$\\
					$v_i$ & \!\!\!\!Magnitude of complex voltage at node $i$ \\
					$\ell_{ij}$ &\!\!\!\!Squared magnitude of complex current on $(i,j)$ \\
					$\mathbb{L}$  &\!\!\!\!Weighted Laplacian matrix of undirected graph\\
					$\lambdaup,\lambdaup_{\max},\lambdaup_{\min}$ &\!\!\!\!Any, maximum, minimum eigenvalues of matrix \\
					$\sigma_{\max}$ &\!\!\!\!Maximum singular value of matrix \\
					\hline
				\end{tabular}
			\end{center}
		\end{small}
	\end{table}

		\vspace{-2mm}
	\section{Network Model and Signal-Taking Voltage Control}\label{sect:model}


	Consider a tree graph $\mathcal{T}=\{\hN \cup\{0\}, \hL\}$ that represents 
	a radial distribution network consisting of $n+1$ nodes and a set $\hL$ of lines between these nodes. Node 0 is the substation node and assumed to have a fixed voltage.  
	For each node $i\in \hN$, denote by $\hL_i \subseteq \hL$ the set of lines on the unique path
	from node $0$ to node $i$, $p_i^c$ and $p_i^g$ the real power consumption and generation,
	and  $q_i^c$ and $q_i^g$  the reactive power consumption and generation, respectively. 
	Let $v_i$ be the magnitude of the complex voltage at node $i$.	For each line $(i, j)\in \hL$, denote by $r_{ij}$ and $x_{ij}$ its resistance and reactance,  and $P_{ij}$ and $Q_{ij}$  the real and reactive power from node $i$ to node $j$, respectively. 
	Let   $\ell_{ij}$ denote the squared magnitude of the complex branch current from node $i$ to node $j$.
	These notations are summarized in Table \ref{table: symbol}. A quantity without subscript is usually a vector with appropriate components defined earlier; 
	e.g., $v := [v_i, i\in\hN]^{\top}, q^g := [q_i^g, i\in \hN]^{\top}$.

	\subsection{Linearized Branch Flow Model} 
	

	We adopt the following branch flow model  introduced in \cite{Baran1989a, Baran1989b} to model a {radial} distribution system:
	\begin{subequations}\label{eq:bfm}
		\begin{eqnarray}
		\hspace{-3mm}P_{ij} &=& p_j^c - p_j^g + \sum_{k: (j,k)\in \hL} P_{jk}+  r_{ij}  \ell_{ij}  \label{p_balance}, \\
		\hspace{-3mm}Q_{ij} &=&  q_j^c-q_j^g + \sum_{k: (j,k)\in \hL} Q_{jk} + x_{ij} \ell_{ij} \label{q_balance},\\
		\hspace{-3mm}v_j^2 &=&  v_i^2 - 2 \left(r_{ij} P_{ij} + x_{ij} Q_{ij}\right) + \left(r_{ij}^2+x_{ij}^2\right) \ell_{ij} \label{v_drop},\\
		\hspace{-3mm}\ell_{ij}v_i^2 &=&   P_{ij}^2 + Q_{ij}^2  \label{currents}.
		\end{eqnarray}
	\end{subequations}
	
	\begin{figure}[ht!]
		\centering
		\includegraphics[scale = 0.2]{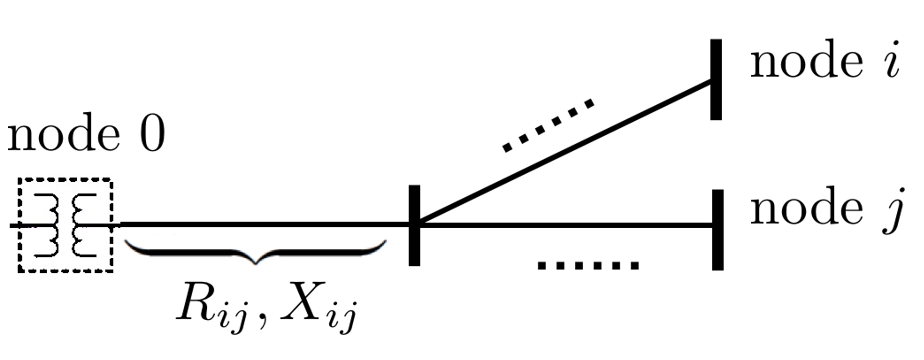}
		\caption{$\hL_{i} \cap \hL_{j}$ for two nodes $i,~j$ and the corresponding mutual voltage-to-power-injection sensitivity factors $R_{ij}$ and $X_{ij}$.}
			\label{fig:RX}
	\end{figure}

	Following \cite{Baran1989c}, we have introduced in \cite{Farivar13,zhou2018reverse} a resistance matrix $R=[R_{ij}]_{n\times n}$ with $R_{ij}:=  \sum_{(h,k)\in \hL_i \cap \hL_j}r_{hk}$ and a reactance matrix $X=[X_{ij}]_{n\times n}$ with $X_{ij}:=   \sum_{(h,k)\in \hL_i \cap \hL_j}x_{hk}$ (see illustration in Fig. \ref{fig:RX}), 
		and derived from \eqref{eq:bfm} a linearized branch flow model: 
	\begin{equation*}  
	v = \overline{v}_0 + R (p^g - p^c) + X(q^g - q^c),
	\end{equation*}
	where $\overline{v}_0 = [v_0, \dots, v_0]^{\top}$ is an $n$-dimensional vector. We assume that $\overline{v}_0, p^c, p^g, q^c$ are given constants, and voltage magnitudes $v := [v_1, \dots, v_n]^{\top}$ and reactive powers $q^g := [q^g_1, \dots, q^g_n]^{\top}$ are the only variables.
	Let  $\tilde{v} = \overline{v}_0 + R (p^g - p^c) - X q^c$, a constant vector.
	For notational simplicity, in the rest of the paper we will ignore the superscript in $q^g$ and write $q$ instead. Then the linearized branch flow model reduces to the following:
	\begin{eqnarray}  \label{model_2}
	v = Xq + \tilde{v}.
	\end{eqnarray}
	We have shown in \cite{Farivar13,zhou2018reverse} that $X$ is positive definite.
	\subsection{The Signal-Taking Volt/Var Control} \label{sect:st}
	The goal of Volt/Var control on a distribution network is to provision reactive power injections $q := [q_1, \dots, q_n]^{\top}$ in order to maintain the node voltages $v := [v_1, \dots, v_n]^{\top}$ to within a tight range around their nominal values $v^{\text{nom}}_i, ~i\in\hN$. In \cite{Farivar13,zhou2018reverse}, we have considered one type of local Volt/Var controls where each node $i$
	makes an individual decision  $q_i(t+1)$ based only on its own voltage 
	$v_i(t)$: 
	\begin{eqnarray} \label{VVC_function}
	q_i(t+1)  \ = \ \left[f_i (v_i(t) - v_i^{\text{nom}} )\right]_{\Omega_i},\quad \forall i\in \hN,
	\end{eqnarray}
	where  $f_i: \mathbb{R}\rightarrow  \mathbb{R}$, 
	$\Omega_i = \left\{q_i  \ | \ {q_i}^{\min}  \leq q_i \leq {q_i}^{\max}\right\}$ the set of feasible reactive power injections at node $i$, and $\left[\ \right]_{\Omega_i}$ denotes the projection onto the set $\Omega_i$. This leads to the following feedback dynamical system for the distribution network:
	\begin{subequations}
		\begin{align}
			v(t) & =  Xq(t) + \tilde{v},
			\label{eq:dynamic1}
			\\
			q(t+1) & =  \left[ f(v(t) - v^{\text{nom}})\right]_\Omega,
			\label{eq:dynamic2}
		\end{align}
	\end{subequations}
where $f : \mathbb{R}^{n} \rightarrow  \mathbb{R}^{n} $ denotes the collection of $[f_i, i\in\hN]^{\top}$, with $\Omega =  \times_{i\in\hN} \Omega_i$. A fixed point of the above dynamical system represents an equilibrium operating point of the network. 
	\begin{definition}~{\em (Definition 2 in \cite{Farivar13,zhou2018reverse})}~
		$\left({v}^*, q^* \right)$ is called an {equilibrium point}, or a {network 
			equilibrium}, if it is a fixed point of \eqref{eq:dynamic1}--\eqref{eq:dynamic2}, i.e., 
		\bqn
		v^* & = &  Xq^* + \tilde{v},
		\\
		q^* & = & \left[ f(v^* - v^{\text{nom}})\right]_{\Omega}.
		\eqn
	\end{definition}
	
	Consider that for each node $i\in\hN$ there is a symmetric deadband 
	$[-\delta_i/2, \delta_i/2]$ around the origin with $\delta_i\geq 0$ in the  control function $f_i$, i.e., $f_i(u_i)=0$ for $u_i\in [-\delta_i/2, \delta_i/2]$. 
	We make the following assumptions:
	\begin{assumption} \label{A1}
		The local Volt/Var control functions $f_i$ are nonincreasing over $\mathbb{ R}$, strictly decreasing, and differentiable in $(-\infty,  - \delta_i/2) \cup (\delta_i/2, \infty)$.
	\end{assumption}
	\begin{assumption}\label{A2}
		The derivative of each control function $f_i$ is bounded, i.e., there exists a finite $\alpha_i$ such that $|f_i'(u_i)| \leq \alpha_i$ for all $u_i$ in the appropriate domain, for all $i\in \hN$.
	\end{assumption}

	
	Define a cost function for each node $i\in\hN$:
	\bqn
	C_i(q_i)  & := & - \int_{0}^{q_i} f_i^{-1}(q) \, dq.
	\eqn
	{Function $f_i$ is decreasing by Assumption \ref{A1}, so is its inverse $f_i^{-1}$. Therefore, $C_i$ is convex by the second order condition of convex functions.
	} Given any $v_i(t)$,  $q_i(t+1)$ in \eqref{eq:dynamic2} is the unique
	solution to an individual decision problem of node $i\in\hN$:
	\bq
	q_i({t+1})  =  \arg\ \ \!\!\! \min_{ q_i \in \Omega_i} 
	\, C_i(q_i) + q_i \left( v_i(t) - v_i^{\text{nom}} \right).
	\label{eq:dynamic2b}
	\eq
	Notice that, in the decision problem \eqref{eq:dynamic2b}, each node $i$ takes the feedback signal $v_i(t)$ as given, and is therefore considered as being {\em signal-taking}. We thus call \eqref{eq:dynamic2} and \eqref{eq:dynamic2b} the {\em signal-taking} voltage control. 

	We have the following results regarding the equilibrium and dynamic properties of the signal-taking voltage control. 
	\begin{theorem}~{\em (Theorem 1 in \cite{Farivar13,zhou2018reverse})}~
		\label{thm:eq}
		Suppose Assumption \ref{A1} holds. Then there exists a unique equilibrium point.
		Moreover, a point $(v^*, q^*)$ is an equilibrium if and only if $q^*$ is the unique
		optimal solution of the following global optimization problem:
		\bq
		\min_{q\in \Omega} & & F(q) = \sum_{i\in\hN}C_i(q_i) + \frac{1}{2}q^{\top} X q + q^{\top} \Delta \tilde{v}
		\label{eq:defminF}
		\eq
		with $\Delta \tilde{v} := \tilde{v} - v^{\text{nom}}$ and $v^* = Xq^* + \tilde{v}$.\hfill$\Box$
	\end{theorem}
	{
		
	With $v= Xq + \tilde{v}$, the objective function $F(q)$ can be equivalently written as:
	\begin{equation*}
		F(q) = \sum_{i\in\hN}C_i(q_i) + \frac{1}{2}(v-v^{\text{nom}})^{\top} X^{-1} (v-v^{\text{nom}}) -\frac{1}{2}\Delta\tilde{v}^{\top}X^{-1}\Delta\tilde{v}, 
	\end{equation*}
	where the last term is a constant. Therefore, the signal-taking local voltage control \eqref{eq:dynamic1}--\eqref{eq:dynamic2} seeks an optimal trade-off between minimizing the cost of reactive power provisioning $\sum_{i\in\hN}C_i(q_i)$ and minimizing the cost of voltage deviation $\frac{1}{2}(v-v^{\text{nom}})^{\top} X^{-1} (v-v^{\text{nom}})$.
	}

	\begin{theorem}~{\em (Corollary 1 in \cite{zhou2018reverse})}~
	\label{thm:con}
	Suppose Assumptions \ref{A1} and \ref{A2} hold. If
	\begin{equation} \label{eq:con}
	\begin{aligned}
		\sigma_{\max}(\mathcal{A}X) < 1,
	\end{aligned}
	\end{equation}
	where $\mathcal{A}\!:=\! \diag \{\alpha_i\}_{i\in\hN}$, then the signal-taking Volt/Var control \eqref{eq:dynamic1}--\eqref{eq:dynamic2} converges to the 
	unique equilibrium point $(v^*, q^*)$. Moreover, it converges exponentially fast to the equilibrium.\hfill$\Box$
\end{theorem}
	
	Theorems \ref{thm:eq}-\ref{thm:con} imply that the dynamical system \eqref{eq:dynamic1}-\eqref{eq:dynamic2}  can be seen as a distributed algorithm for solving the optimization problem \eqref{eq:defminF}. 
	
	\section{The Signal-Anticipating Voltage Control}\label{sect:posa}
	
	As discussed in Section \ref{sect:model}, in the equivalent decision problem \eqref{eq:dynamic2b}, each node takes the feedback signal $v_i(t)$ as given, making the local Volt/Var control \eqref{eq:dynamic2} a {\em signal-taking} control. However, a node $i$ may be able to learn or infer the impact of its own decision $q_i$ on the feedback signal $v_i$, i.e., node $i$ knows $v_i$ as a function of $q_i$ (see equation \eqref{model_2}), and takes it into consideration when making the control decision on reactive power, which we call the {\em signal-anticipating} voltage control. With the signal-anticipating control, node $i\in\hN$ will decide its reactive power output according to the following optimization problem: 
	\begin{eqnarray}
\hspace{-2.5mm}q_i(t+1) 	\hspace{-2mm}&=& \hspace{-2mm}\arg\  \min_{q_i\in\Omega_i}  C_i(q_i)+q_i (X_{ii} q_i +\hspace{-2mm} \sum_{j\in\hN\backslash \{i\}} \hspace{-2mm}X_{ij} q_j(t) + \Delta \tilde{v}_i). 
	\label{eq:sa}
	\end{eqnarray}
	To see the difference from the signal-taking control, notice that \eqref{eq:dynamic2b} can be written as:
	\begin{eqnarray}
	q_i(t+1) \hspace{-2mm}&=& \hspace{-2mm}\arg\  \min_{q_i\in\Omega_i}  C_i(q_i)+q_i \Big(\sum_{j\in\hN} X_{ij} q_j(t) + \Delta \tilde{v}_i\Big). 
	\label{eq:st}
	\end{eqnarray}
	The signal-anticipating voltage control makes the interaction among the nodes  a game. 
	
	\begin{definition}\label{def:vcg}
		A {\em voltage control game} is defined as a triple $\mathcal{G}=\{ \hN, ( \Omega_i)_{i\in\hN},  (h_i)_{i\in\hN}\}$, with a set $\hN$ of players (nodes), node $i\in\hN$ strategy space $\Omega_i$, and cost function $h_i(q)= C_i(q_i)+q_i (\sum_{j\in\hN} X_{ij} q_j + \Delta \tilde{v}_i)$. 
	\end{definition}
	
	Let $q_{-i}=[q_1, \cdots, q_{i-1}, q_{i+1}, \cdots, q_N]^{\top}$ denote the reactive powers at all nodes other than $i$, and represent $q$ as $[q_i, q_{-i}^{\top}]^{\top}$. The signal-anticipating voltage control \eqref{eq:sa} can be written as 
	\begin{eqnarray}
	q_i(t+1) = \arg\  \min_{q_i\in\Omega_i} ~h_i(q_i, q_{-i}(t)), ~i\in\hN,
	\label{eq:br}
	\end{eqnarray}
	and is therefore the best response algorithm for the voltage control game $\mathcal{G}$ \cite{Fud91}.
	
	\begin{remark}
	The signal-anticipating control \eqref{eq:sa} can be rewritten as $$q_i(t+1) = \arg\  \min_{q_i\in\Omega_i} ~C_i(q_i)+q_i (X_{ii} q_i  - X_{ii} q_i (t) + v_i(t)  - {v}_i^{nom}),$$
	which can be computed using the {current} reactive power provisioning $q_i(t)$ and {the locally measured} voltage $v_i(t)$  at node $i$ {if $X_{ii}$ is known.} {Note that the voltage-to-power injection sensitivity factor $X_{ii}=\frac{\partial v_i}{\partial q_i}$, so it is reasonable to assume node $i$ can estimate it. Moreover, $X_{ij}$'s are physical parameters of the distribution network. So, it is also reasonable for the distribution system operator to estimate them and ``inform'' the nodes.} Therefore, the signal-anticipating control is a local control like the signal-taking control. 
		\end{remark}
	
	\subsection{Equilibrium}
	We now analyze the Nash equilibrium of the voltage control game. A vector $q^a$ of reactive powers is a Nash equilibrium if, for all nodes $i\in\hN$,  $h_i(q_i^a, q_{-i}^a)\leq h_i(q_i, q_{-i}^a)$ for all $q_i\in\Omega_i$.  We see that the Nash equilibrium is a set of reactive powers for which no node has the incentive to change unilaterally. 
	
	\begin{lemma}
		A Nash equilibrium $q^a$ of the voltage control game $\mathcal{G}$ is a fixed point (or equilibrium) of the signal-anticipating voltage control \eqref{eq:sa}, and vice versa.
	\end{lemma}
	
	\begin{proof}
		The result follows from the fact that the signal-anticipating voltage control is the best response algorithm for the voltage control game $\mathcal{G}$; see equation \eqref{eq:br}. 
	\end{proof}
	
	Considering the function $W: \Omega \rightarrow \mathbb R$:
	\bqn
	W(q) =\sum_{i\in\hN}\left(C_i(q_i) + \frac{1}{2}X_{ii}q_i^2\right)+ \frac{1}{2}q^{\top} X q + q^{\top} \Delta \tilde{v}
	\eqn
	and the global optimization problem:
	\bq
	\min_{q\in \Omega} & & W(q). 
	\label{eq:defminW}
	\eq
	We have the following result { on the relation between the Nash equilibrium of voltage control game and the optimum of problem $\eqref{eq:defminW}$.}
	\begin{theorem}
		\label{thm:neq}
		Suppose Assumption \ref{A1} holds.  Then there exists a unique Nash equilibrium for the voltage control game $\mathcal{G}$.
		Moreover a point $q^a$ is a Nash equilibrium if and only if it is the unique
		optimum of \eqref{eq:defminW}.\hfill$\Box$
	\end{theorem}
	
	\begin{proof}
		First, notice that the problem  \eqref{eq:defminW} is strictly convex. So, the first order optimality condition for  \eqref{eq:defminW} is both sufficient and necessary; and moreover,  \eqref{eq:defminW} has a unique optimum. Second, notice that the first order optimality condition is just the fixed point condition of the best response algorithm \eqref{eq:br}. The existence and uniqueness of the optimum of \eqref{eq:defminW} then implies that of the Nash equilibrium $q^a$. 
	\end{proof}
	
	\subsection{Dynamics}
	We now study the dynamic properties of the signal-anticipating voltage control \eqref{eq:sa}, i.e., the best response algorithm \eqref{eq:br}. 
			\begin{theorem} \label{anticipating_converge}
			Suppose Assumptions \ref{A1} and \ref{A2} hold. If 
			\begin{equation} \label{condtion1}
			\begin{aligned}
			\sigma_{\max}(\mathcal{B}\Bar{X}) < 1,
			\end{aligned}
			\end{equation}
			where $\mathcal{B} := \diag \{\beta_i\}_{i\in\hN}$ with $\beta_i = \frac{1}{1/\alpha_i + 2X_{ii}}$ and $\Bar{X}$ is the resulting matrix by setting diagonal elements of $X$ to 0,
			then the signal-anticipating voltage control \eqref{eq:sa} converges to the unique Nash equilibrium of the voltage control game $\mathcal{G}$. Moreover, it converges exponentially fast to the equilibrium.\hfill$\Box$
		\end{theorem}
		\begin{proof}
			Notice that the solution to \eqref{eq:sa} for each $i \in \hN$ satisfies the following equations:
			\begin{equation*}
			\begin{aligned}					\tilde{q_{i}}(t+1) &= f_{i}\Big(2X_{ii}\tilde{q_{i}}(t+1) + \sum_{j \in \hN \backslash \{i\}}X_{ij}q_{j}(t) + \Delta \tilde{v}_i \Big),\\
			q_i(t+1) &= [\tilde{q_i}(t+1)]_{\Omega_i},
			\end{aligned}
			\end{equation*}
			where the first equation is equivalent  to 
			\begin{equation*}
			\begin{aligned}
			\underbrace{f_i^{-1}(\tilde{q}_i(t+1)) - 2X_{ii}\tilde{q}_i(t+1)}_{g_i^{-1}(\tilde{q}_i(t+1))} = \sum_{j \in \hN \backslash \{i\}}X_{ij}q_{j}(t) + \Delta \tilde{v}_i. 
			\end{aligned}
			\end{equation*}
			Here we define $g_i^{-1}(x) := f_i^{-1}(x) - 2X_{ii}x.$	
			Since $g_i^{-1}$ is strictly decreasing,  its inverse $g_i$ exists. 
			So, the signal-anticipating control \eqref{eq:sa} can be written as: 
			\begin{equation}\label{new_control}
			q_i(t+1) = \Bigg[g_i\Big(\sum_{j \in \hN \backslash \{i\}}X_{ij}q_{j}(t) + \Delta \tilde{v}_i\Big)\Bigg]_{\Omega_i}. 
			\end{equation}
			Based on the definition of $\Bar{X}$, we rewrite \eqref{new_control} in the vector form of 
			\begin{equation}
			\begin{aligned} \label{vector_new_control}
			q(t+1) = \left[g(\Bar{X}q(t) + \Delta \tilde{v})\right]_{\Omega},
			\end{aligned}
			\end{equation}
			where $g(x) := [g_1(x_1),\cdots,g_n(x_n)]^{\top }$.
			
			Under Assumption 2, $f_i^{'} \in [-\alpha_i, 0]$ in the appropriate domain. Thus, $(f_i^{-1})^{'} \leq -\frac{1}{\alpha_i}$ and $(g_i^{-1})^{'} \leq -\frac{1}{\alpha_i} - 2X_{ii}$.  Therefore,  $g_i$ has bounded gradient: $g_i^{'} \in [-\frac{1}{1/\alpha_i + 2X_{ii}},0] = [-\beta_i,0]$, by definition of $\beta_i$. 
			
			Although $g$ is non-smooth at some points, we can show that it is Lipschitz continuous, i.e., there exists a constant $L$ such that $\|g(x) - g(y)\|_2 \leq L\|x-y\|_2, \forall x,y$. To see this, 	without loss of generality, we assume $x_i \geq y_i$. If both $x_i$ and $y_i$ are in $(-\infty, -\frac{\delta_i}{2}]$ or in $[\frac{\delta_i}{2}, + \infty),$ by the mean value theorem we have $|g_i(x_i) - g_i(y_i)| \leq \beta_i|x_i-y_i|.$ If both are in $[-\frac{\delta_i}{2}, \frac{\delta_i}{2}]$, we have $0 = |g_i(x_i) - g_i(y_i)| \leq \beta_i |x_i-y_i|.$ If $x_i \in [\frac{\delta_i}{2}, + \infty)$ and $y_i \in [-\frac{\delta_i}{2}, \frac{\delta_i}{2}],$ $|g_i(x_i) - g_i(y_i)| = |g_i(x_i) - g_i(\frac{\delta_i}{2})| \leq \beta_i|x_i-\frac{\delta_i}{2}| \leq \beta_i |x_i - y_i|$, where the first inequality follows from the mean value theorem. Similarly, we can show that $|g_i(x_i) - g_i(y_i)| \leq \beta_i |x_i - y_i|$ holds under other situations too. Therefore,
			\begin{equation*}
			\begin{aligned}
			\|g(x)-g(y)\|_2 \leq \|\mathcal{B}(x-y)\|_2 \leq \sigma_{\max}(\mathcal{B})\|x-y\|_2.    
			\end{aligned}
			\end{equation*}
			
			Based on \eqref{vector_new_control} and the non-expansiveness property of projection operator, given any feasible $q,q'$, we can conclude
			\begin{equation*}
			\begin{aligned}
			&~~~\left\| \left[g(\Bar{X}q + \Delta \tilde{v})\right]_{\Omega} - \left[g(\Bar{X}q' + \Delta \tilde{v})\right]_{\Omega} \right\|_2\\
			& \leq \|g(\Bar{X}q + \Delta \tilde{v}) - g(\Bar{X}q' + \Delta \tilde{v})\|_2\\
			&\leq \|\mathcal{B}\Bar{X}(q-q')\|_2 \leq \sigma_{\max}(\mathcal{B}\Bar{X}) \| q- q'\|_2. 
			\end{aligned}
			\end{equation*}
			If condition \eqref{condtion1} is satisfied, the signal-anticipating control \eqref{vector_new_control} is a contraction mapping. This implies that $q(t)$ (and thus $v(t)$) converges exponentially fast to the unique Nash equilibrium of the voltage control game $\mathcal{G}$.
		\end{proof}
	
	{
	Recall that the signal-anticipating voltage control  \eqref{eq:sa} is the best response algorithm of the voltage control game. In general, the best response is not a converging strategy, but that is not case in our problem. On the other hand, the convergence condition \eqref{condtion1} is also expected:  if the control is cautious enough in adjusting reactive power (i.e., small enough $\alpha_i$), it will converge. It also shows that if $X_{ii}$ is large enough, the signal-anticipating control will converge too. This is consistent with our intuition that the signal-anticipating behavior will help with convergence, as each node adapts to the expected voltage rather than the current voltage. Indeed, as will be seen in the next subsection, the convergence condition \eqref{condtion1} is less restrictive than that for the signal-taking control. 
	}

	{Note that the convergence condition \eqref{condtion1} requires calculating the maximum singular value of a matrix, which may not be easy when the size of the network is large.}  
	We thus develop a sufficient condition for \eqref{condtion1}, which is easier to verify in practice. 
	\begin{corollary}
		Suppose Assumptions \ref{A1} and \ref{A2} hold. If
		\begin{equation}\label{sufficient}
			\max_{i \in \hN} \beta_i\cdot\max_{i\in \hN} \sum_{j \in \hN} \bar{X}_{ij}  < 1,
		\end{equation}
		then the signal-anticipating voltage control \eqref{eq:sa} converges exponentially fast to the unique Nash equilibrium of the voltage control game $\mathcal{G}$. 
	\end{corollary}
	\begin{proof}
		Based on H\"{o}lder's inequality for matrix norms, we have 
		\begin{equation} \label{holder}
			\sigma_{\max}(\mathcal{B}\bar{X}) = \|\mathcal{B}\bar{X}\|_2 \leq \sqrt{\|\mathcal{B}\bar{X}\|_1\|\mathcal{B}\bar{X}\|_{\infty}} .
		\end{equation} 
		A sufficient condition for \eqref{condtion1} is given by 
			$\|\mathcal{B}\bar{X}\|_1 < 1$ and $\|\mathcal{B}\bar{X}\|_{\infty} < 1$, 
		which are satisfied if the condition \eqref{sufficient} holds.
		\end{proof}
	{
	
\subsection{Comparison of Convergence Conditions}\label{sect:con-comp}

By taking into account the impact of its own decision, the signal-anticipating voltage control adapts to the expected voltage rather than the current voltage (that results from the current reactive power provisioning), which expects to result in better convergence property. This is indeed the case as can be seen from the following result. 

		
		\begin{theorem} \label{theorem:5}
		The convergence condition \eqref{condtion1} for the signal-anticipating voltage control is less restrictive than the condition \eqref{eq:con} for the signal-taking voltage control, i.e.,
			\begin{equation}\label{equ:17}
				\sigma_{\max}(\mathcal{B}\Bar{X}) < \sigma_{\max}(\mathcal{A}X). 
			\end{equation}
		\end{theorem}
		
		\begin{proof}
			We prove Equation \eqref{equ:17} in two steps. We first show that $\sigma_{\max}(\mathcal{B}X) < \sigma_{\max}(\mathcal{A}X)$. Recall that $\mathcal{A} := \diag \{\alpha_i\}_{i\in\hN}$ and  $\mathcal{B} := \diag \{\beta_i\}_{i\in\hN}$ with $\beta_i = \frac{1}{1/\alpha_i + 2X_{ii}}$.  Define a  matrix $\mathcal{C} = \diag \{e_i\}_{i\in\hN}$ with $e_i = \frac{1}{1 + 2X_{ii}\alpha_{i}}$, we have $\mathcal{C}\mathcal{A} = \mathcal{B}$. Notice that $0< e_i <1$, so $\sigma_{\max}(\mathcal{C}) <1$. Therefore, 			
			\begin{equation*}
				\sigma_{\max}(\mathcal{B}X) = \sigma_{\max}(\mathcal{C}\mathcal{A}X) \leq \sigma_{\max}(\mathcal{C}) \sigma_{\max}(\mathcal{A}X) < \sigma_{\max}(\mathcal{A}X). 
			\end{equation*}
			
			We next show that $\sigma_{\max}(\mathcal{B\bar{X}}) < \sigma_{\max}(\mathcal{B}X).$
			Notice that $0 \prec (\mathcal{B\bar{X}})^{\top}\mathcal{B\bar{X}} \prec (\mathcal{B}X)^{\top}\mathcal{B}X$, where ``$\prec$'' denotes the element-wise inequality. Thus, Frobenius norms $\|((\mathcal{B\bar{X}})^{\top}\mathcal{B\bar{X}})^m\|_F < \|((\mathcal{B}X)^{\top}\mathcal{B}X)^m\|_F$ for $m>0$, from which we have $\|((\mathcal{B\bar{X}})^{\top}\mathcal{B\bar{X}})^m\|_F^{1/m} < \|((\mathcal{B}X)^{\top}\mathcal{B}X)^m\|_F^{1/m}.$ By Gelfand's formula \cite{reed2012methods}, 
\begin{eqnarray}
\nonumber \lambdaup_{\max}((\mathcal{B\bar{X}})^{\top}\mathcal{B\bar{X}}) &=&  \lim_{m \rightarrow \infty} \|((\mathcal{B\bar{X}})^{\top}\mathcal{B\bar{X}})^m\|_F^{1/m}  \\
\nonumber &<&  \lim_{m \rightarrow \infty} \|((\mathcal{B}X)^{\top}\mathcal{B}X)^m\|_F^{1/m}\\
\nonumber &=& \lambdaup_{\max}((\mathcal{B}X)^{\top}\mathcal{B}X), 
\end{eqnarray}
i.e., $\sigma_{\max}(\mathcal{B\bar{X}}) < \sigma_{\max}(\mathcal{B}X)$. 

Combining the two steps above, we obtain $\sigma_{\max}(\mathcal{B}\Bar{X}) < \sigma_{\max}(\mathcal{A}X)$. 
		\end{proof}
	
	By \eqref{equ:17}, if the condition \eqref{eq:con} holds, then the condition \eqref{condtion1} holds too. So, the signal-anticipating voltage control entails better convergence property. For example, given the same distribution network, it can converge under control functions with larger slopes than the signal-taking control.	
}

	\section{The Price of Signal-Anticipat{ion}}\label{pose}
		
The signal-taking and signal-anticipating behaviors in local voltage control are analogous to the price-taking and price-anticipating behaviors in economics. It is well-known that in a competitive market with price-taking customers the system achieves an efficient equilibrium and in an oligopolistic market with price-anticipating (or strategic) customers the system usually incurs efficiency loss; see, e.g., \cite{MWG,nisan2007algorithmic}. Similarly, the signal-taking behavior leads to an efficient equilibrium while the signal-anticipating behavior may result in efficiency loss, in term of the global cost function $F(q)$. In this section, we study such an efficiency loss. 
	
	\subsection{Price of Signal Anticipat{ion}}
	
	We define \textbf{price of signal anticipat{ion}} (PoSA) to characterize the impact of the signal anticipating behavior in local voltage control in particular and in distributed control in general.
	\begin{definition}
		The PoSA is defined as the gap in cost (or efficiency loss) between the network equilibrium $q^{*}$ of the dynamics \eqref{eq:dynamic2b} and the Nash equilibrium $q^{a}$ of the dynamics  \eqref{eq:sa}, i.e.,
		\begin{equation}
		\text{PoSA} = F(q^{a}) - F(q^{*}). \label{PoSA}
		\end{equation}
	\end{definition}

	We aim to investigate how PoSA scales with the size, topology, and heterogeneity of the distribution network. Such results will be insightful to {understanding the signal-anticipating behavior} and designing mechanisms to mitigate its impact if necessary.

	In this paper, we will focus on a special case where each node $i \in \mathcal{N}$ has a quadratic cost functions $C_{i}(q_{i}) = \frac{1}{2}y_{i}q_{i}^{2}$ with $y_{i} > 0.$ Quadratic cost functions are widely used in market models for the power system. We further assume for simplicity that there is no constraint in reactive power, i.e., $q_{i}^{\min} = - \infty, q_{i}^{\max} = \infty, i \in \mathcal{N}$. But the results are expected to extend to more general settings.
	
	Define the cost matrix $Y = \diag\{y_{i}\}_{i\in\hN}$ and denote by $D$ the diagonal part of $X$. The network equilibrium $q^{*}$ arising from the signal-taking behavior solves
	\begin{eqnarray}
	\min_{q} ~~~F(q) = \frac{1}{2}q^{\top}(X+Y)q + q^{\top}\Delta \tilde{v},
	\end{eqnarray} 
	i.e., $q^{*} = -(X+Y)^{-1}\Delta\tilde{v}$; whereas the Nash equilibrium $q^{a}$ arising from the signal-anticipating behavior solves
	\begin{eqnarray}
	\min_{q} ~~~W(q) = \frac{1}{2}q^{\top}(X+D+Y)q + q^{\top}\Delta \tilde{v}, 
	\end{eqnarray}
	i.e., $q^{a} = -(X+D+Y)^{-1}\Delta \tilde{v}$.
	 
	
	\begin{lemma}\label{lemma: posa}
	    	Given the reactance matrix $X$ and the cost matrix $Y$, the PoSA can be reformulated in a compact form as
	\begin{eqnarray}
	\text{PoSA} =F(q^{a}) - F(q^{*})= \frac{1}{2}\Delta \tilde{v}^{\top}\Pi\Delta \tilde{v}, \label{final_posa}
	\end{eqnarray}
	where $\Pi = (X+D+Y)^{-1}D(X+Y)^{-1}D(X+D+Y)^{-1}$.
	\end{lemma}
	{
	\begin{proof} 
	We have
		\begin{eqnarray} 
		&&\!\!\!\!\!\!F(q^{a}) -F(q^{*}) \nonumber\\
		&=& \!\!\!\!\!\!\frac{1}{2}\Delta \tilde{v}^{\top} \Big(\underbrace{(X+D+Y)^{-1}(X+Y)(X
			+D+Y)^{-1}\!\!\!-(X+D+Y)^{-1}}_{\textbf{(a)}} \nonumber\\
		&&~~~~~~\!\!\!+\underbrace{(X+Y)^{-1} - (X+D+Y)^{-1}}_{\textbf{(b)}}\Big)\Delta\tilde{v}. \label{equ:a+b}
		\end{eqnarray}
		By the Woodbury's formula \cite{golub2012matrix},
		\begin{equation*} \label{woodbury}
		(X+D+Y)^{-1} = (X+Y)^{-1} - (X+Y)^{-1}(D^{-1}+(X+Y)^{-1})^{-1}(X+Y)^{-1}.
		\end{equation*}
		Substituting the above into \textbf{(a)} and \textbf{(b)}, we obtain: 
		\begin{equation}
		\begin{aligned}
		\textbf{(a)}&=-(X+Y)^{-1}(D^{-1}+(X+D+Y)^{-1})^{-1}(X+D+Y)^{-1},\nonumber\\
		\textbf{(b)}&= (X+D+Y)^{-1}D(X+Y)^{-1}. \nonumber\\
		\end{aligned}
		\end{equation}
		Combining the above two equations, we have
		\begin{equation}
		\begin{aligned}
		\textbf{(a)}+\textbf{(b)}&=  (X+D+Y)^{-1}\underbrace{(D-(D^{-1}+(X+Y)^{-1})^{-1})}_{\textbf{(c)}}(X+Y)^{-1}.\nonumber
		\end{aligned}
		\end{equation}
			By the Woodbury's formula, the second term in \textbf{(c)}
		\begin{equation*}
		(D^{-1} + (X+Y)^{-1})^{-1} = D - D(X+D+Y)^{-1}D,
		\end{equation*}
		so we have $\textbf{(c)} = D(X+D+Y)^{-1}D$. Thus, we get
		\begin{equation}
		\begin{aligned}
		\textbf{(a)}+\textbf{(b)}&=(X+D+Y)^{-1}D(X+D+Y)^{-1}D(X+Y)^{-1}, \nonumber
		\end{aligned}
		\end{equation}
		
		i.e., $PoS\!A = \frac{1}{2}\Delta \tilde{v}^{\top}\Pi\Delta \tilde{v}$. 
	\end{proof}
}

	Notice that $\Pi \succ 0$, and thus PoSA is always positive for nonzero $\Delta \tilde{v}$ (which is actually the initial voltage deviation). Since PoSA is quadratic in $\Delta \tilde{v}$, without normalization PoSA can be arbitrarily large. We therefore investigate a normalized, worst-case PoSA with respect to the norm of $\Delta\tilde{v}$, defined as:
	\begin{eqnarray}
	\text{PoSA}_{\max}(X,Y) = \sup_{\Delta\tilde{v}} \frac{\text{PoSA}}{\Delta \tilde{v}^{\top}\Delta \tilde{v}} = \frac{1}{2}\sup_{\Delta\tilde{v}} \frac{\Delta \tilde{v}^{\top}\Pi\Delta \tilde{v}}{\Delta \tilde{v}^{\top}\Delta \tilde{v}}.
	\end{eqnarray}
	It follows that 
	\begin{eqnarray}
	\text{PoSA}_{\max}(X,Y)\! =\! \frac{1}{2}\lambdaup_{\max}(\Pi), \label{eq:nprosa}
	\end{eqnarray}
	where $\lambdaup_{\max}$ denotes the maximum eigenvalue, and can be achieved by the eigenvector of $\Pi$ corresponding to the maximum eigenvalue. 
	
	We will next characterize the $\text{PoSA}_{\max}$. 
	Before doing that, first notice that, if there are multiple subtrees at node $0$, the voltage controls at these subtrees are independent as node $0$ has a fixed voltage. Therefore, without loss of generality, for the rest of the paper we consider the network where node $0$ has only one direct child node.

{
\begin{remark}
We had originally considered $\frac{ F(q^a) - F(q^\star) } { F(q^\star) }$ or $\frac{ F(q^a) } { F(q^\star) }$ as the metric of efficiency loss, but the specific structure of the cost function in our problem does not allow a good analytical characterization of them. The metrics of the ratio type such as the price of anarchy was motivated by the competitive analysis of approximation or online algorithms where {\em competitive ratio} is used to characterize the performance of the algorithm. An analytical characterization of the competitive ratio is usually possible only if the problem has certain nice structure. Besides the competitive ratio, another typical metric used to characterize the performance of approximation or online algorithms is {\em regret} that is the difference in the achieved cost  of the algorithm and certain optimal cost. Our definition of the PoSA is a metric of the regret type.  
\end{remark}
}

	\subsection{Upper Bound on $\text{PoSA}_{\max}$}
	
	We now characterize the $\text{PoSA}_{\max}$ by deriving an upper bound for $\lambdaup_{\max}(\Pi)$. 
	Let $d = \max_{i\in \mathcal{N}} X_{ii} $, $y = \min_{i\in \mathcal{N}} y_{i}$, and $\lambdaup$, $\lambdaup_{\min}$ denote an arbitrary and the minimum eigenvalue respectively. 
	
	\begin{lemma} \label{lem:12}  An upper bound on eigenvalues of matrix $D(X+D+Y)^{-1}$ is given by: 
		\begin{equation}
		\begin{aligned}
		\lambdaup (D(X+D+Y)^{-1}) \leq \frac{d}{\lambdaup_{\min}(X)+d+y}.
		\end{aligned}
		\end{equation}
	\end{lemma}
	\begin{proof}
		By properties of matrix eigenvalue, we have 
		\begin{equation*}
		\begin{aligned}
		X + Y &\succeq \left(\lambdaup_{\min}(X)+y\right)I \succeq (\lambdaup_{\min}(X)+y)\frac{D}{d}, 
		\end{aligned}
		\end{equation*}
		where $I$ denotes the identity matrix. Equivalently,
		\begin{equation*}
		\begin{aligned}
		X+D+Y \succeq  \left(\frac{\lambdaup_{\min}(X)+y+d}{d}\right)D,
		\end{aligned}
		\end{equation*}
		which leads to
		\begin{equation*}
				D(X+D+Y)^{-1} \preceq \frac{d}{\lambdaup_{\min}(X)+d+y}I.
		\end{equation*} 
	\end{proof}

	By Lemma \ref{lem:12}, we derive an upper bound on $\lambdaup_{\max}(\Pi)$. 
	\begin{theorem} \label{theorem:1}
		The largest eigenvalue of $\Pi$ is upper bounded by:
		\begin{equation}\label{upper bound}
		\begin{aligned}
		\lambdaup_{\max}(\Pi)  \leq \lambdaup_{\max}((X+Y)^{-1}).
		\end{aligned}
		\end{equation} 
		That is, the $\text{PoSA}_{\max}$ is upper bounded by $\frac{1}{2}\lambdaup_{\max}((X+Y)^{-1})$.
	\end{theorem}
	\begin{proof} Notice that, given two matrices $A$ and $B$, $AB$ is similar to $BA$ and $\lambdaup(AB) \leq \lambdaup(A)\lambdaup(B)$~\cite{horn2012}. 
	We thus have 
		\begin{align}
		\lambdaup_{\max}(\Pi) &= \lambdaup_{\max}(D^{2}(X+D+Y)^{-2}(X+Y)^{-1}) \nonumber\\
		&\leq \lambdaup_{\max}(D^{2}(X+D+Y)^{-2})\lambdaup_{\max}((X+Y)^{-1}) \nonumber\\
		&\leq \frac{d^{2}}{(\lambdaup_{\min}(X)+d+y)^{2}}\lambdaup_{\max}((X+Y)^{-1}) \nonumber \\ 
		&\leq \lambdaup_{\max}((X+Y)^{-1}),
		\end{align}
		where the second inequality follows from Lemma \ref{lem:12}. 
	\end{proof}
	
	By \cite{horn1954}, $\lambdaup_{\min}(X) \leq\min_{i\in\hN} X_{ii}=X_{11}$. For a (large) tree network with large depth, $d=\max_{i\in\hN} X_{ii}$ scales linearly with the depth of the tree, and thus $d>> \lambdaup_{\min}(X) $ and $d>> y$. We see that $\frac{d^{2}}{(\lambdaup_{\min} (X) +d+y)^{2}} \rightarrow 1$ for large networks.

	\begin{remark}
	By Weyl's inequality \cite{weyl1912}, we have
	\begin{equation*}
	\lambdaup_{\max}((X+Y)^{-1}) = \frac{1}{\lambdaup_{\min}(X+Y)} 
	\leq \frac{1}{\lambdaup_{\min}(X) + y} 
	\leq \frac{1}{y}.
	\end{equation*}   
	We see that the $\text{PoSA}_{\max}$ is upper bounded by a constant independent of the size of the network, and the ``average'' $\text{PoSA}_{\max}$ per node goes to zero as the size of the network approaches to infinity. This is a desirable property as it states that the $\text{PoSA}_{\max}$ will not be arbitrarily large, no matter what the size of the network is. 
	
\end{remark}

	\subsection{Tightness of the Bound}
	We now investigate the tightness of the upper bound given by Theorem \ref{theorem:1}  under different conditions. To that end, we first give a lower bound on $\lambdaup_{\max}(\Pi)$, and then numerically investigate the tightness of the upper and lower bounds under different network settings. 
	\begin{theorem} \label{lower}
		The largest eigenvalue of $\Pi$ is lower bounded by:
		\begin{equation}
			\lambdaup_{\max}(\Pi) \geq \lambdaup_{\max}((X+Y)^{-1}-2(X+Y+D)^{-1}) . 
		\end{equation}
	\end{theorem}
	\begin{proof}
		{Recall from equation \eqref{equ:a+b} that $\Pi$ is the sum of  two parts \textbf{(a)} and \textbf{(b)}. By re-arranging the terms in these two parts, 
		}we have 	
		\begin{equation*}
		\begin{aligned}
		\lambdaup_{\max}(\Pi) &= \lambdaup_{\max}(\underbrace{(X+D+Y)^{-1}(X+Y)(X+D+Y)^{-1}}_{{\textbf{(d)}}} \nonumber\\
		&~~~~~~~~+(X\!+\!Y)^{-1}\! -\! 2(X\!+\!D\!+\!Y)^{-1})\nonumber\\	
		&\geq \lambdaup_{\min}({\textbf{(d)}})+\lambdaup_{\max}((X\!+\!Y)^{-1}\! -\! 2(X\!+\!D\!+\!Y)^{-1}) \nonumber\\
		&\geq \lambdaup_{\max}((X\!+\!Y)^{-1}\! -\! 2(X\!+\!D\!+\!Y)^{-1}),
		\end{aligned}
		\end{equation*}
		where the first inequality is due to Weyl's inequality \cite{weyl1912} and the second inequality results from the fact that ${\textbf{(d)}} \succeq 0$.
	\end{proof}

    By Theorems  \ref{theorem:1} and \ref{lower}, we have 
    \begin{equation*}
	    	\lambdaup_{\max}((X+Y)^{-1}\!\!-2(X+Y+D)^{-1}) \leq \lambdaup_{\max}(\Pi) \leq \lambdaup_{\max}((X+Y)^{-1}).
    \end{equation*}
    By Weyl's inequality, we have 
	    $\lambdaup_{\max}(X+Y)^{-1} - \lambdaup_{\max}((X+Y)^{-1}-2(X+Y+D)^{-1}) \leq 2\lambdaup_{\max}((X+D+Y)^{-1})$. So the gap between the lower and upper bounds is bounded by $2\lambdaup_{\max}((X+D+Y)^{-1})$. Although this bound on the gap can be large for general matrix, due to the special structure of the reactance matrix $X$ 
	    our numerical experiments show that it is very small under most of the  settings (see below). This means that the upper and lower bounds on the $\text{PoSA}_{\max}$ are tight.

	\begin{figure}[ht!]
	\center{
		\includegraphics[scale = 0.7]{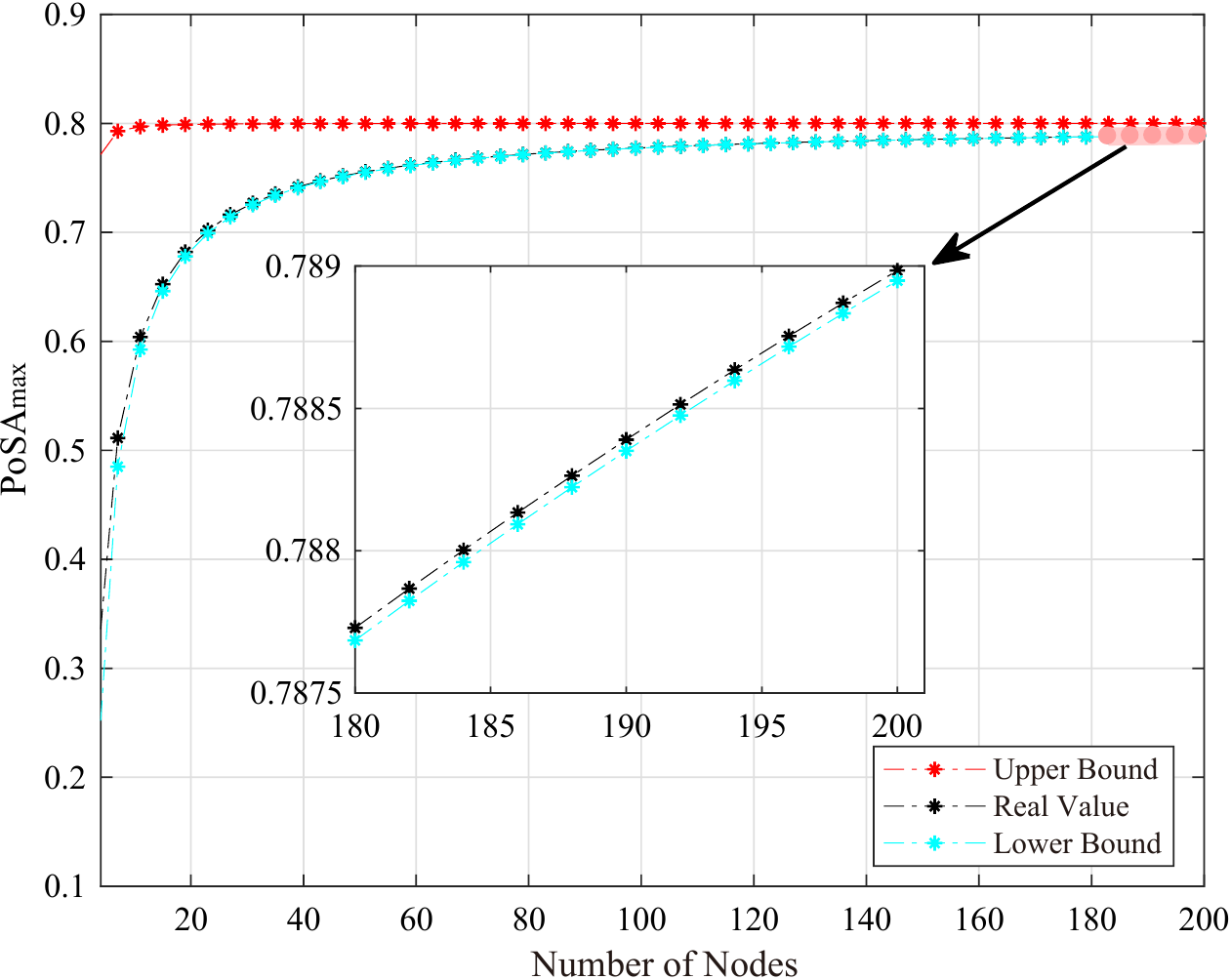}
		\caption{Bounds on $\text{PoSA}_{\max}$ for linear network with $x_{ij} = 1,~Y = I$.}
		\label{fig:fig3}}
	\end{figure}
	
	\begin{figure}[ht!]
	\center{
		\includegraphics[scale = 0.7]{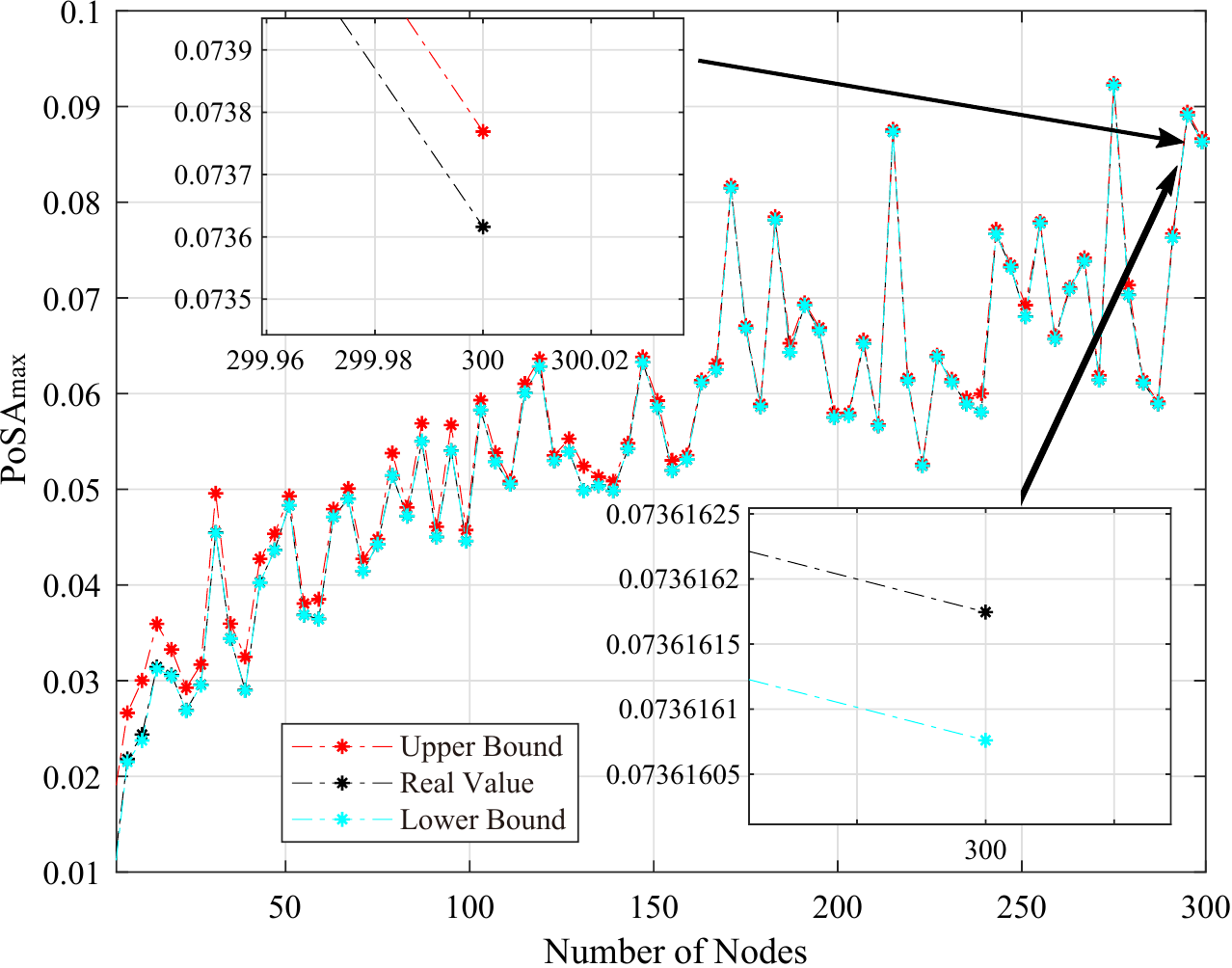}
		\caption{Bounds on $\text{PoSA}_{\max}$ for linear network with randomly chosen $x_{ij} \in (0,200],~y_i \in (0,100].$ Each sample is repeated 10 times.}
		\label{fig:fig4}}
	\end{figure}

Fig.~\ref{fig:fig3} shows the comparison between actual values and lower/upper bounds of the $\text{PoSA}_{\max}$ for a homogeneous linear network with different sizes. 
We see that the gap between the lower and upper bounds approaches to 0 as the size of the network increases. 
Fig. \ref{fig:fig4} shows the comparison for a heterogeneous linear network with randomly chosen parameters. Again, we see the bounds on the $\text{PoSA}_{\max}$ are tight when the size of the network is large. 


	\begin{figure}[ht!]
	\center{
	\includegraphics[scale = 0.7]{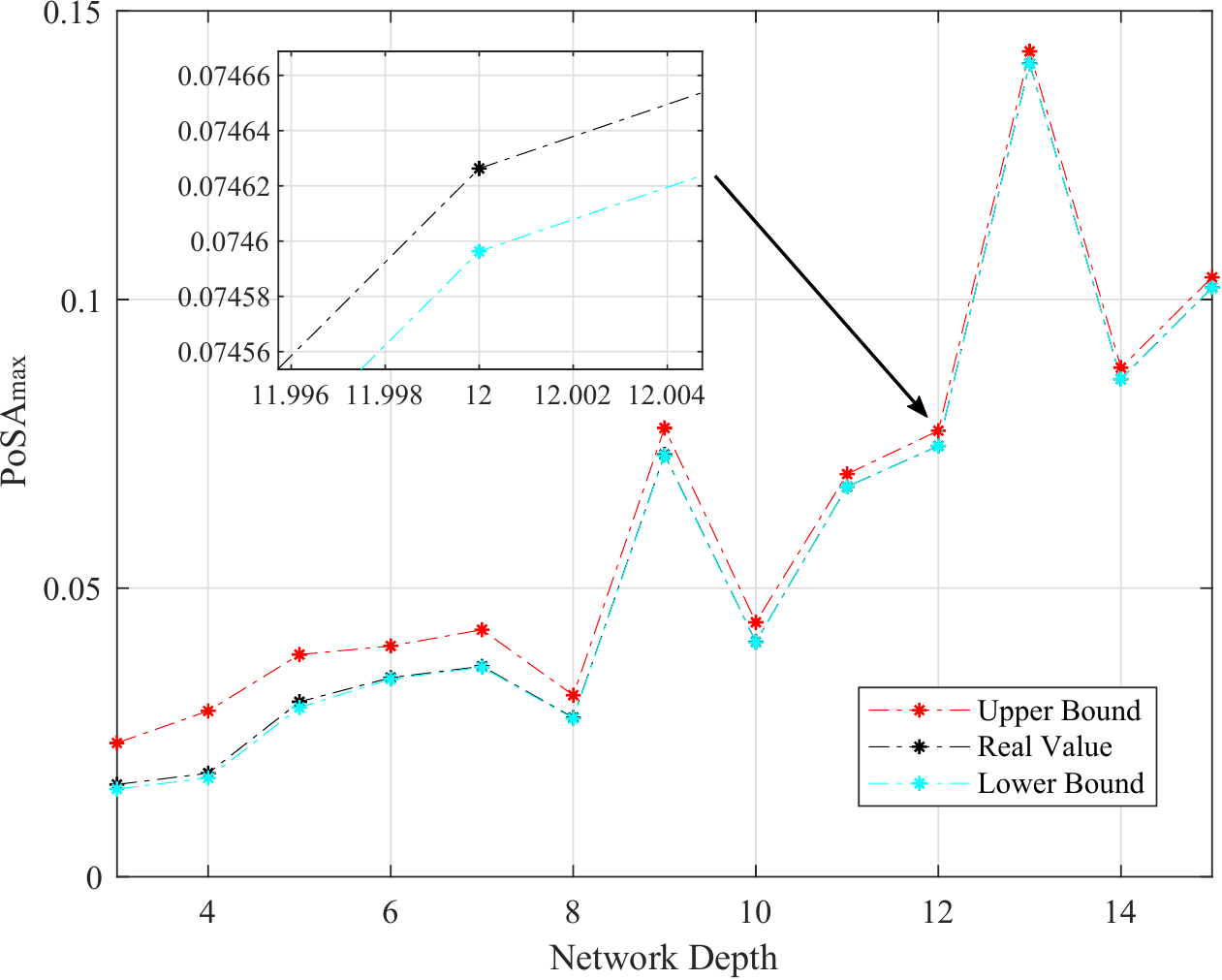}
	\caption{Bounds on $\text{PoSA}_{\max}$ for random tree with degree distribution $\mathbb{P} = \{0.5,0.5\}$ and randomly chosen $x_{ij} \in (0,200],~y_i \in (0,100]$. Each sample is repeated 10 times.}
	\label{fig:fig5}}
    \end{figure}
	
	\begin{figure}[ht!]
	\center{
		\includegraphics[scale = 0.7]{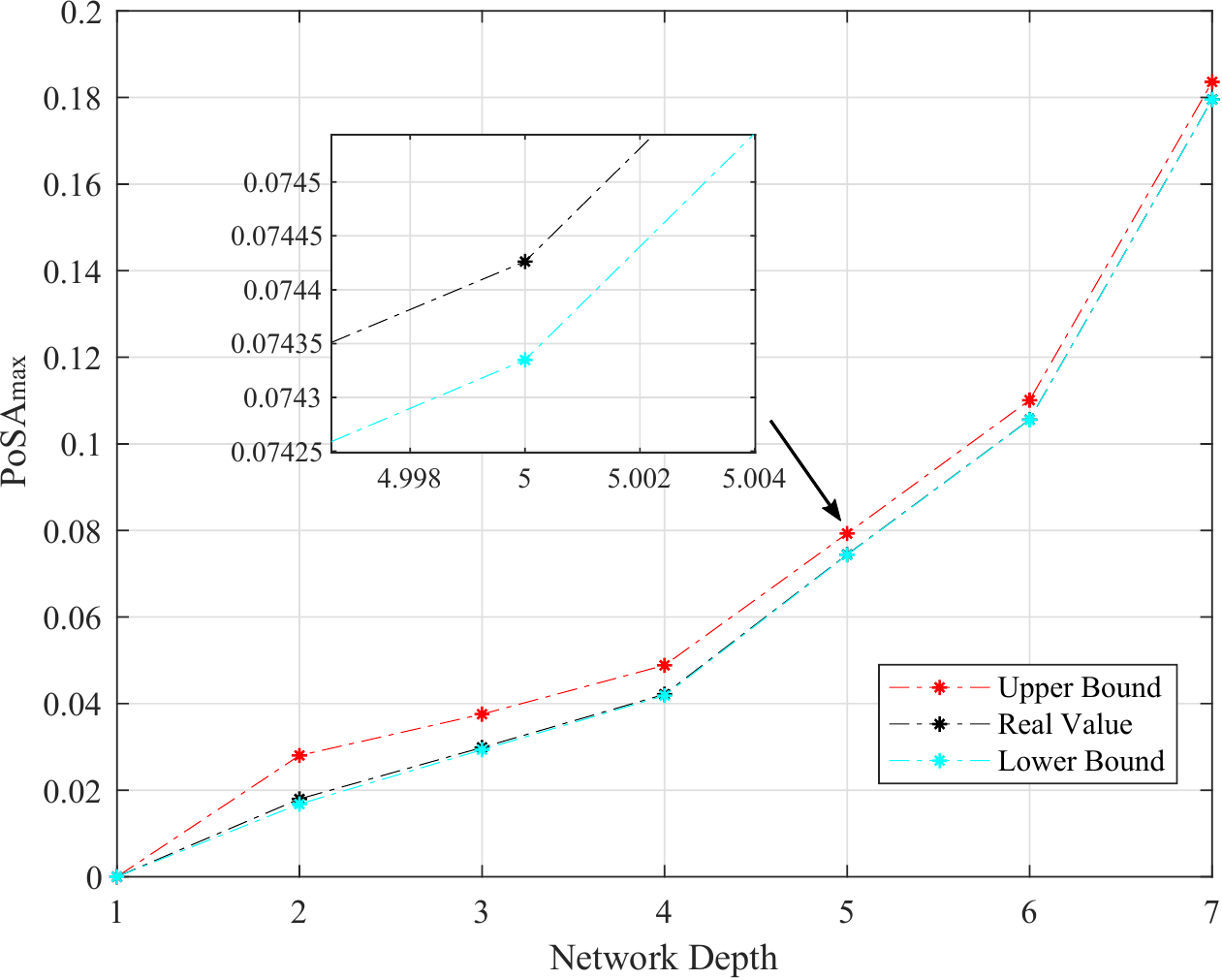}
		\caption{Bounds in $\text{PoSA}_{\max}$ for ternary tree with degree distribution $\mathbb{P} = \{0.3,0.3,0.4\}$ and randomly chosen $x_{ij} \in (0,200],~y_i \in (0,100].$ Each sample is repeated 10 times.}
		\label{fig:fig6}}
	\end{figure}

	We next evaluate the tightness of the bounds on random trees. The random trees are generated according to the degree distribution $\mathbb{P}$ in the number of children. For example, $\mathbb{P} = \{0.5,0.5\}$ means a node will have 0.5 probability of having one child and 0.5 probability of having two children. Moreover, the reactance of the power line and the cost coefficient of the node  are randomly chosen. Fig.~\ref{fig:fig5} shows the bounds for a random binary tree generated by distribution $\mathbb{P} = \{0.5,0.5\}$ and of up to the depth of 15 (totaling about 1400 nodes). Fig.~\ref{fig:fig6} shows the bounds for a ternary tree  generated by distribution $\mathbb{P} = \{0.3,0.3,0.4\}$. We see that the lower and upper bounds on the $\text{PoSA}_{\max}$ are tight when the size of network is large, and give a good approximation to the actual  $\text{PoSA}_{\max}$.

	Notice that all our numerical experiments show that the lower bound on the $\text{PoSA}_{\max}$ is tighter than the upper bound. Nonetheless, we choose the upper bound instead of the lower bound for analysis of $\text{PoSA}_{\max}$ for two reasons. First, we consider the worst-case $\text{PoSA}_{\max}$, so the upper bound is more appropriate. Second, even though we do not have a closed form for either bound, the upper bound is simpler to compute. 	
	Moreover, as will be seen in the next subsection,  we can derive a closed form for the upper bound for linear networks. 
	
	
	\subsection{Case Study: Linear Network}
	
	In this subsection, we investigate the $\text{PoSA}_{\max}$ for the linear network as shown in Fig.~\ref{fig: linear}. 
	
	\begin{figure}[ht]
		\centering
		\includegraphics[width = 7cm,height= 0.65cm]{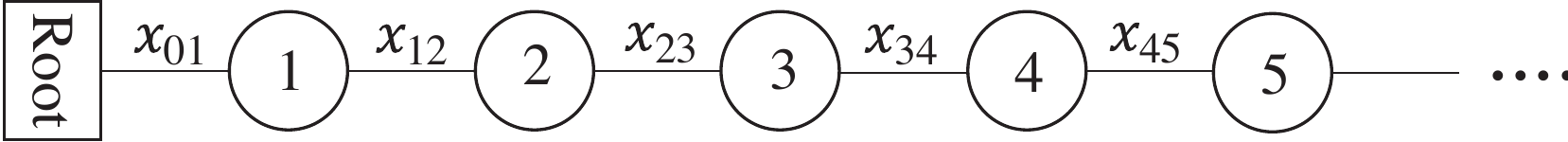}
		\caption{Linear Network}
		\label{fig: linear}
		\vspace{0mm}
	\end{figure} 
	
	For a linear network, the reactance matrix $X$ is given by
	\begin{equation*}
	\begin{aligned}
	X = \begin{bmatrix}
	x_{0,1}&\cdots&\cdots&x_{0,1}\\
	\vdots&\sum_{i=0}^{2}x_{i,i+1}&\cdots&\sum_{i=0}^{2}x_{i,i+1}\\
	\vdots&\vdots&\ddots&\vdots\\
	x_{0,1}&\sum_{i=0}^{2}x_{i,i+1}&\cdots&\sum_{i=0}^{n-1}x_{i,i+1}
	\end{bmatrix}, 
	\end{aligned}
	\end{equation*}
	where for clarity we write $x_{ij}$ as $x_{i,j}$. 
	By Lemma \ref{Laplacian} in the Appendix, $X^{-1}$ is a symmetric tridiagonal matrix: 
	\begin{equation*}\label{inverse}
	\begin{aligned}
	X^{-1} = 
	\begin{bmatrix} 
	\frac{x_{0,1}+x_{1,2}}{x_{0,1}x_{1,2}}&-\frac{1}{x_{1,2}}&0&\cdots&\cdots\\
	-\frac{1}{x_{1,2}}&\frac{x_{1,2}+x_{2,3}}{x_{1,2}x_{2,3}}&-\frac{1}{x_{2,3}}&0&\cdots\\
	0&-\frac{1}{x_{2,3}}&\frac{x_{2,3}+x_{3,4}}{x_{2,3}x_{3,4}}&-\frac{1}{x_{3,4}}&0\\
	\vdots&0&\ddots&\ddots&\ddots\\
	\vdots&\vdots&0&-\frac{1}{x_{n-1,n}}&\frac{1}{x_{n-1,n}} 
	\end{bmatrix}.
	\end{aligned}
	\end{equation*}
	We can have an analytical form for the eigenvalues of $X$ when all the $x_{i,i+1}$ take the same value.
	\begin{lemma} \label{special case}
		If all the line reactances $x_{i,i+1}$ take the same value, denoted by $a$, the $k$-th largest eigenvalue $\lambdaup_{k}$ of matrix $X^{-1}$ is given by
		\begin{equation}\label{eigens}
		\begin{aligned}
		\lambdaup_{k} = \frac{2}{a} + \frac{2}{a}\cos\frac{2k\pi}{2n+1}, k = 1,2,\cdots,n.
		\end{aligned}
		\end{equation}
	\end{lemma}
	\begin{proof}
		We have 
		\begin{equation*}
		\begin{aligned}
		X^{-1} = \begin{bmatrix}
		\frac{2}{a}&-\frac{1}{a}&0&\cdots&\cdots\\
		-\frac{1}{a}&\frac{2}{a}&-\frac{1}{a}&0&\cdots\\
		0&-\frac{1}{a}&\frac{2}{a}&-\frac{1}{a}&0\\
		\vdots&0&\ddots&\ddots&-\frac{1}{a}\\
		\vdots&\vdots&0&-\frac{1}{a}&\frac{1}{a}
		\end{bmatrix}.
		\end{aligned}
		\end{equation*}
		By Theorem 1 of \cite{Yueh2005}, $X^{-1}$ has eigenvalues 
		\begin{eqnarray}
		\lambdaup_{k} = \frac{2}{a} + \frac{2}{a}\cos\frac{2k\pi}{2n+1}, k = 1,2,\cdots,n. \nonumber
		\end{eqnarray}	
	\end{proof}
	
	The following result is immediate. 
	\begin{theorem}\label{thm:bfl}
		For a linear network of size $n$ with all the line reactances equal to $a$, an upper bound of  $\text{PoSA}_{\max}$ is given by
	\begin{equation*} \label{linear_bound}
	\begin{aligned}
	\text{PoSA}_{\max}(X,Y)	&\leq
	\frac{(an)^{2}}{\left(\frac{a}{2+2\cos\frac{2\pi}{2n+1}}+an+y\right)^{2}\left(y+\frac{a}{2+2\cos\frac{2\pi}{2n+1}}\right)}.	
	\end{aligned}
	\end{equation*} 
	\end{theorem}
	\begin{proof}
	By the inequality \eqref{upper bound} and the fact that $\lambdaup_{\min}(X+Y) \geq \lambdaup_{\min}(X) + y$, we have 
	\begin{align}\label{temp}
		\text{PoSA}_{\max}(X,Y) \leq  \frac{d^{2}}{(\lambdaup_{\min}(X)+d+y)^{2}(\lambdaup_{\min}(X) + y)}.
	\end{align}
	By Lemma \ref{special case}, $\lambdaup_{\min}(X) = \frac{a}{2+2\cos\frac{2\pi}{2n+1}}$ when $k=1$. Substitute it into \eqref{temp} to get the result.
	\end{proof}
	
	\begin{figure}[ht!]
	\center{
		\includegraphics[trim={0mm 0 0mm 0mm},clip,scale = 0.32]{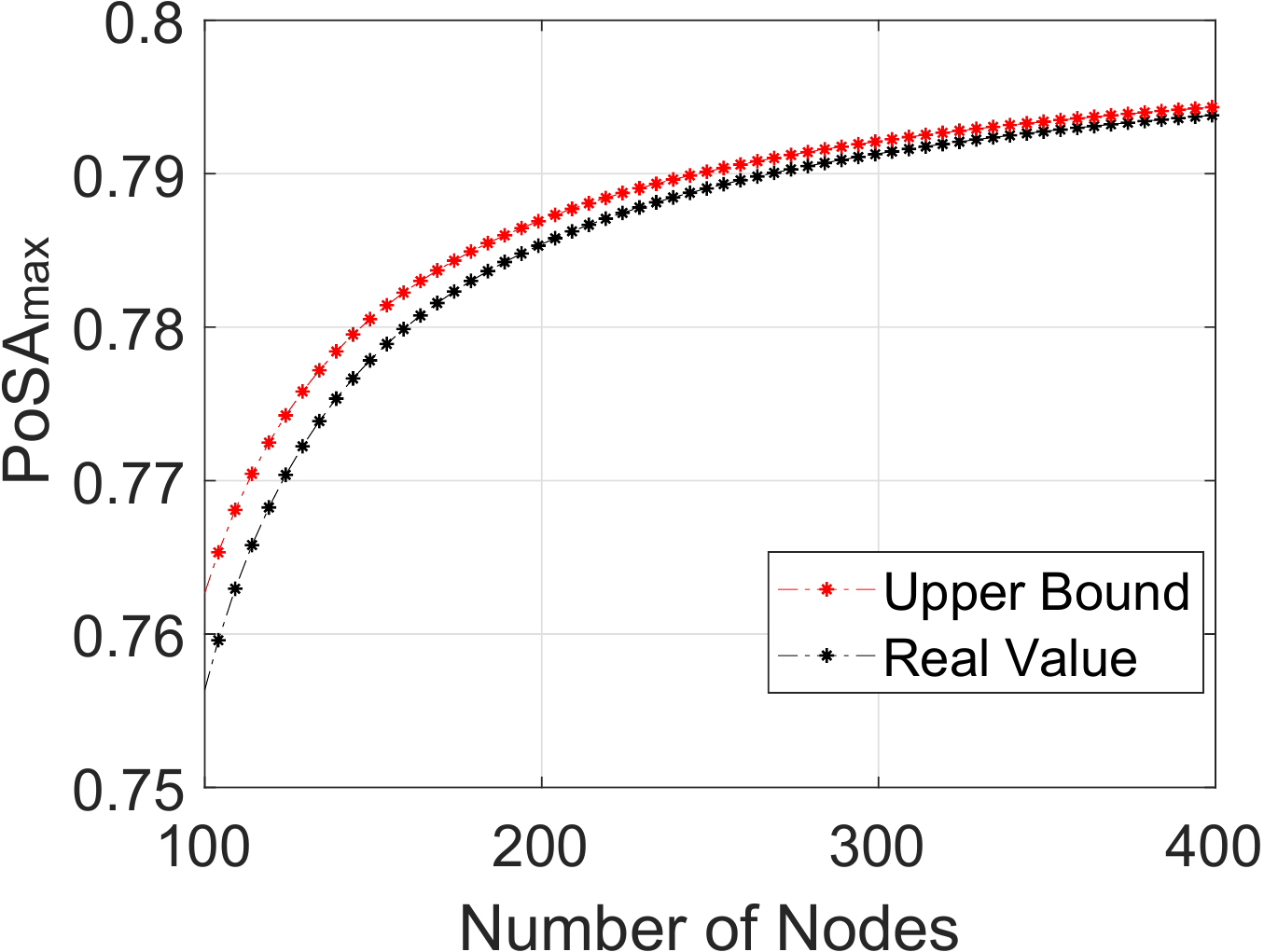}\hspace*{-3pt}
				\includegraphics[trim={0mm 0 0mm 0mm},clip,scale = 0.32]{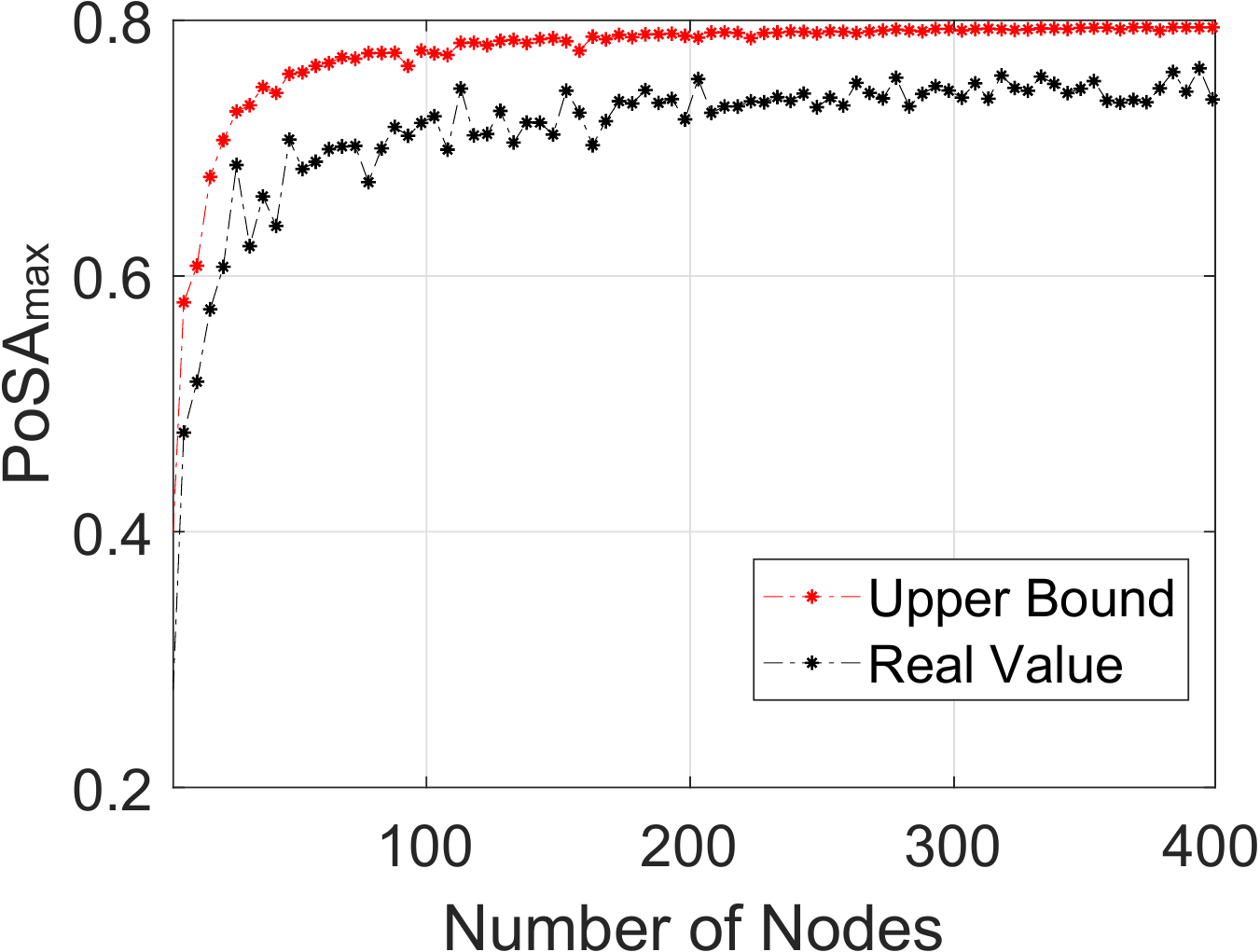} 
		\caption{Upper bound on  $\text{PoSA}_{\max}$ for linear network with $x_{ij} = 1,~y_i = 1$ (left) and with $x_{ij} = 1$ and randomly chosen $y_i \in [1,1.5]$ (right).} 
		\label{fig:fig8}}
	\end{figure}
	
	Notice that, if the eigenvectors corresponding to the minimum eigenvalues of $X$ and $Y$ are aligned, the bound given in Theorem \ref{thm:bfl} will be tighter. This can happen when $Y \approx \gamma I$ for certain $\gamma>0$, as shown in Fig.~\ref{fig:fig8} (left). Otherwise, if the cost coefficients $y_i$ of nodes span a large range, the corresponding eigenvectors to the minimum eigenvalues will unlikely be aligned and the bound will be less tight. This can be seen from Fig.~\ref{fig:fig8} (right) for a network with randomly chosen cost coefficients. 
	

	Lemma \ref{special case} considers the special case with all line reactances $x_{ij}$ taking the same value.  We next consider a more general setting with $x_{ij} \in [a,b]$ for certain $b\geq a >0$. 
	\begin{lemma}\label{theorem2}
		If the line reactances $x_{ij} \in [a,b]$ in the linear network, we can bound the $k$-th largest eigenvalue $\lambdaup_{k}$ of $X$ as follows:
		\begin{equation*}
		\begin{aligned}
		\lambdaup_{k}^{lower}	\leq \lambdaup_{k} \leq \lambdaup_{k}^{upper},~~k = 1,~\cdots,~n,
		\end{aligned}
		\end{equation*}
		where 
		\begin{equation*}
		    \lambdaup_{k}^{lower} = \frac{a}{2+2\cos\frac{2(n-k+1)\pi}{2n+1}}	~\text{ and }~
		   \lambdaup_{k}^{upper} = \frac{b}{2+2\cos\frac{2(n-k+1)\pi}{2n+1}}.
		\end{equation*}
	\end{lemma}
	\begin{proof}
	First, notice that for a linear network, if any of the $x_{ij}$ increases, all eigenvalues of $X$ will not decrease. To see this, suppose that one $x_{ij}$ increases by $\epsilon$. Then, the new reactance matrix $X^{\prime}$ can be written as
		\begin{equation*}
		\begin{aligned}
		X^{\prime} = X +\underbrace{ \begin{bmatrix}
			0&\cdots&\cdots&\cdots\\
			\vdots&\ddots&\cdots&\cdots\\
			\vdots&\vdots&\epsilon&\epsilon\\
			\vdots&\vdots&\epsilon&\ddots
			\end{bmatrix}}_{P}. \nonumber
		\end{aligned}
		\end{equation*}
		Notice that $P$ is a positive semi-definite matrix, So, the eigenvalues of $X^{\prime}$ will not be smaller than those of $X$, respectively. 
		

	If $x_{ij}$ randomly takes value in $[a,b]$, denote by $X(a)$ the reactance matrix with all $x_{ij} = a$ and $X(b)$ the matrix with all $x_{ij} = b$,  we have the following inequality to lower and upper bound $k$-th largest eigenvalue $\lambdaup_{k}(X)$ with the corresponding $k$-th largest eigenvalues of $X(a)$ and $X(b)$: 
	\begin{eqnarray}
	\lambdaup_{i}(X(a)) \leq \lambdaup_{i}(X) \leq \lambdaup_{i}(X(b)). \nonumber
	\end{eqnarray}
	The result then follows from Lemma \ref{special case}. 
	\end{proof}
	
	 By Lemma \ref{theorem2}, the following result is immediate. 
	 \begin{theorem}
	 For a linear network with the line reactances $x_{ij} \in  [a,b]$, the  $\text{PoSA}_{\max}$ is bounded by 
	\begin{equation*} 
	\begin{aligned}
	PoSA_{\max}(X,Y) &\leq \frac{d^{2}}{(\lambdaup_{n}^{\text{lower}}+d+y)^{2}(y+\lambdaup_{n}^{\text{lower}})},\\
	\end{aligned}
	\end{equation*} 
	\end{theorem}
	
	Similarly, the tightness of the above bound depends on the heterogeneity of line reactances $x_{ij}$ and cost coefficient $y_i$.
	
	{ The above analytical and numerical results show that the $\text{PoSA}_{\max}$, i.e., the efficiency loss from the signal-anticipating voltage control, is upper bounded by a constant, independent of the size of the network, and the average $\text{PoSA}_{\max}$ per node goes to zero as the size of the network increases. This is desirable as it means that the $\text{PoSA}_{\max}$ will not be arbitrarily large, no matter what the size of the network is, and no mechanism is needed to mitigate the signal-anticipating behavior.}

{		
\begin{remark}	
		 if it happens that the efficiency loss of the signal-anticipating control is large, under the engineering perspective where the signal-anticipation comes from the limited/partial information at individual nodes, the distribution system operator may just tell each nodes to be signal-taking; and under the economic perspective where the signal-anticipation comes from the nodes being self-interested and strategic, the  distribution system operator may design proper incentive mechanisms such as pricing to mitigate the impact of the signal-anticipating behavior and to induce an efficient systemwide outcome. 
\end{remark}
}

	\begin{figure}
		\centering
		\includegraphics[scale=0.265]{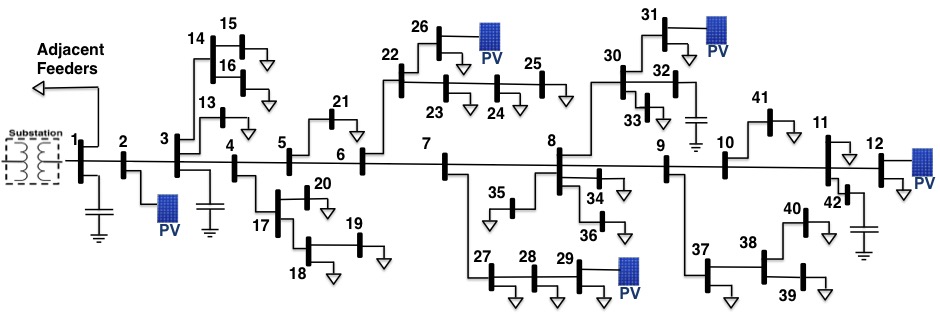}
		\caption{Circuit diagram for SCE distribution system.}
		\label{42bus}
	\end{figure}
{
	\section{Numerical Experiments with Real-world Circuit}
	In this section, we consider a 42-bus distribution feeder of Southern California Edison (SCE) as shown in Fig.~\ref{42bus}, to exam numerically the efficiency loss and convergence of the signal-anticipating voltage control. 
	
	As shown in Fig. \ref{42bus}, bus 1 is the substation (root bus) and five PV generators are integrated at buses 2, 12, 26, 29 and 31. Table \ref{data} describes the network data including the line impedance, the peak MVA demand of loads, and the capacity of the PV generators. As we aim to study the Volt/Var control through the PV inverters, all shunt capacitors are assumed to be off. We implement the piece-wise linear droop control function of the IEEE 1547.8 Standard, with the same slope $\alpha_i$  and deadband $[-\delta_i/2, \delta_i/2]$ at all inverters:
	\begin{equation} \label{eq:droop}
		f_i(v_i) = -\alpha_{i}[v_i - \delta_i/2]^{+} + \alpha_{i}[-v_i - \delta_i/2]^{+}.
	\end{equation}
This corresponds to the following cost function in reactive power provisioning:
\begin{equation}
		\begin{aligned}
			C_i(q_i) = \begin{cases}
			\frac{y_i}{2}q_i^2 - \frac{\delta_i}{2}q_i, & \text{if}\ q_i \leq 0, \\
			\frac{y_i}{2}q_i^2 + \frac{\delta_i}{2}q_i, & \text{if}\ q_i \geq 0,
			\end{cases} 
		\end{aligned}
	\end{equation}
where $y_i = 1/\alpha_i$. All the experiments are run with a full AC power flow model using MATPOWER \cite{zimmerman2011matpower} instead of its linear approximation. 

As we study a real-world distribution circuit, we cannot scale it to different sizes as we do in Section \ref{pose}. Instead, we will consider different values of $\alpha_i$ or equivalently $y_i$. We first exam the efficiency loss of the signal-anticipating voltage control. Fig.~\ref{fig:loss_posa} shows the PoSA$_{\text{max}}$ value versus the $y_i$ value with a deadband $\delta_i=0.02$~p.u. and capacity constraints of the PVs. Fig.~\ref{fig:loss_posan} shows  the PoSA$_{\text{max}}$ value versus the $y_i$ value for an ``ideal'' situation with no deadband (i.e., $\delta_i=0$) or PV capacity constraints. We see that the efficiency loss is very small and decreases with increasing $y_i$, as expected. 


			\begin{figure}[ht!]
		\center{
			\includegraphics[scale = 0.6]{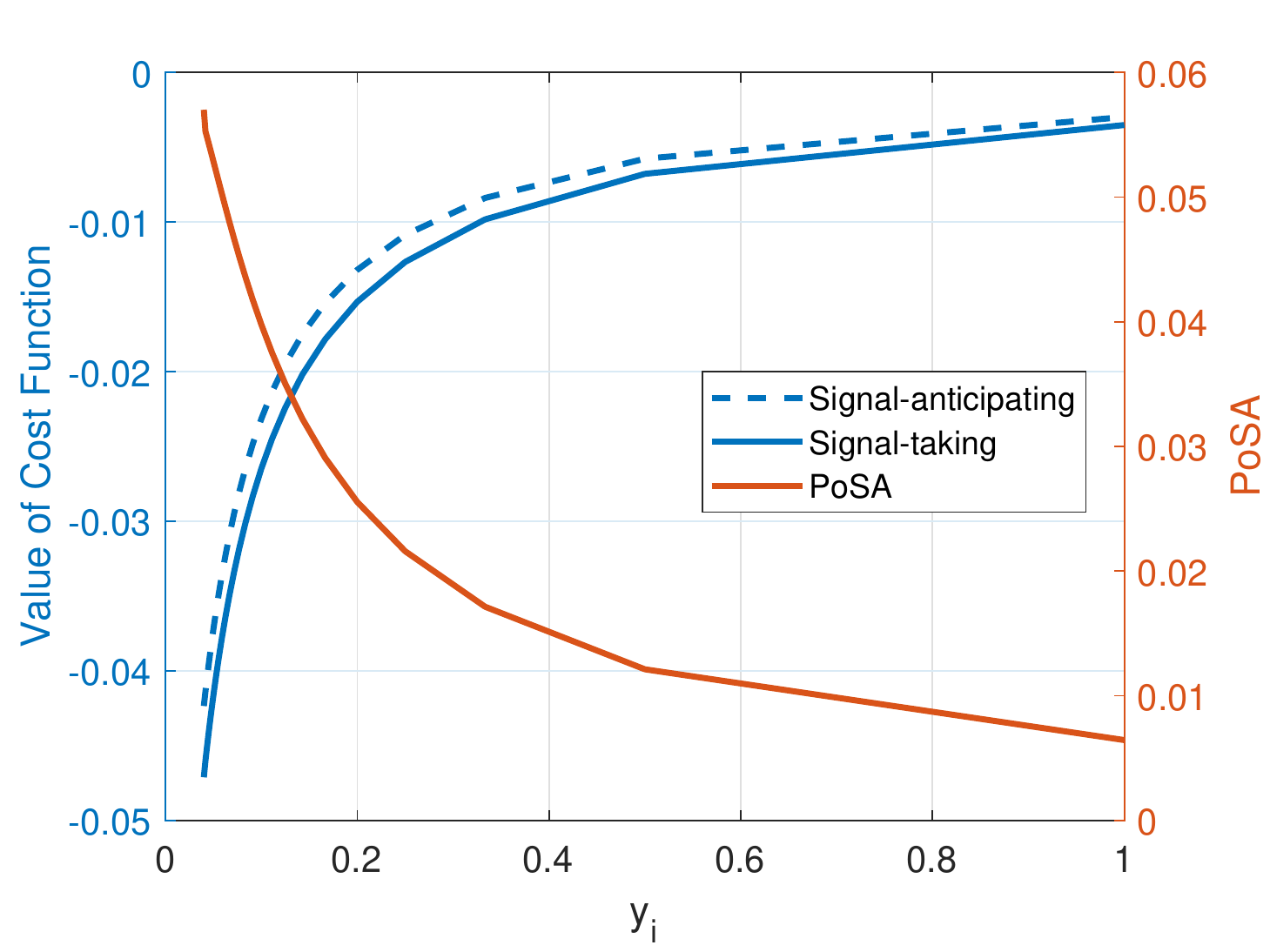}
			\caption{ Cost function values comparison and the price of signal-anticipation with a deadband $\delta_i=0.02$~p.u. and PV capacity constraints. 
			}
			\label{fig:loss_posa}}
	\end{figure}
	
	\begin{figure}[ht!]
		\center{
			\includegraphics[scale = 0.6]{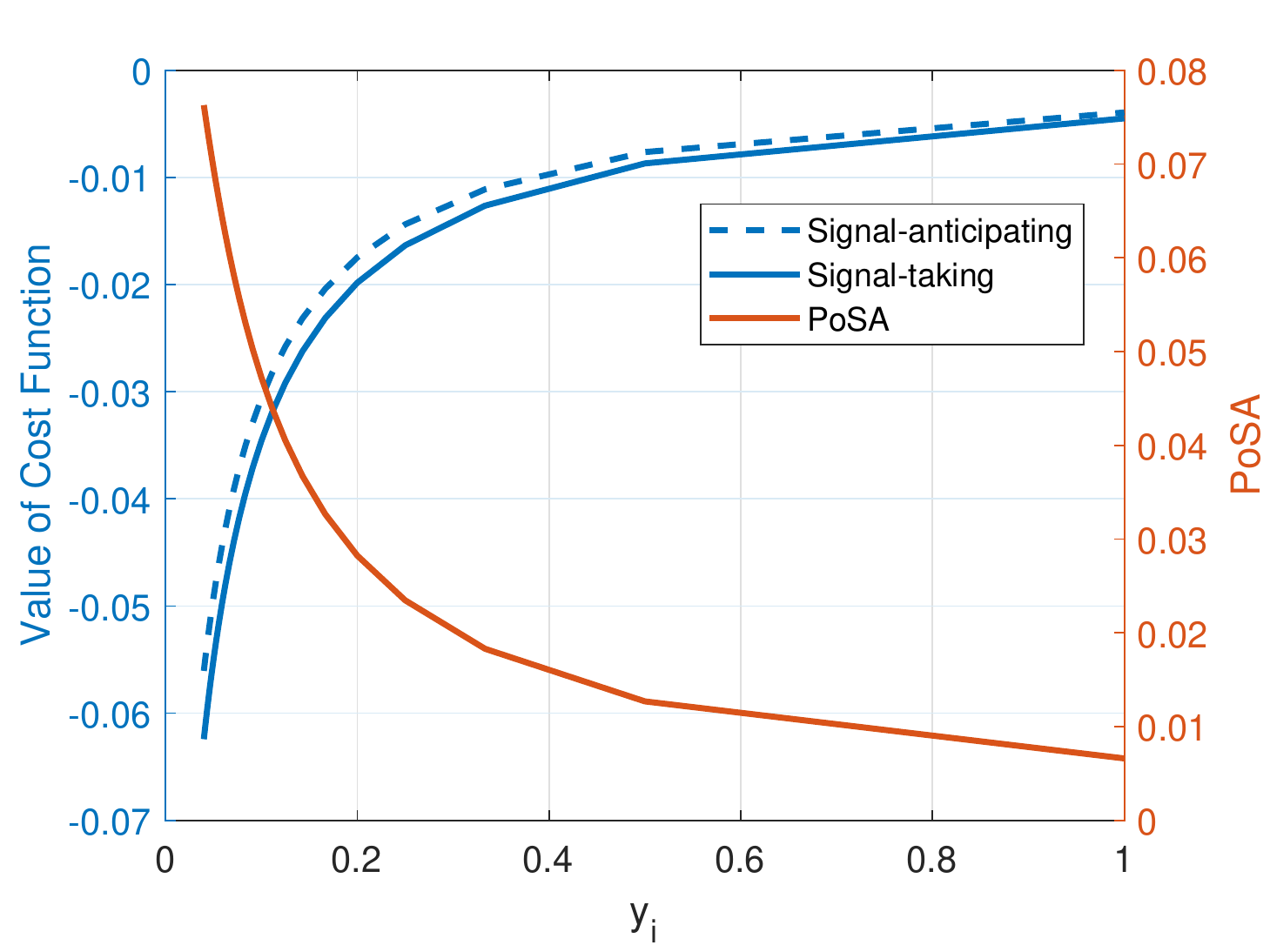}
			\caption{ Cost function values comparison and the price of signal-anticipation without deadband or PV capacity constraints.}
			\label{fig:loss_posan}}
	\end{figure}

%
	
	  	
  	We next exam the dynamical property of the signal-anticipating voltage control. Fig.~\ref{fig:exper} shows the evolution of voltages for different $\alpha_i$ values under the signal-taking and signal-anticipating voltages controls with a deadband $\delta_i=0.02$~p.u. and PV capacity constraints.  Fig.~\ref{fig:expern} shows that for the controls 
without deadband or PV capacity constraints. As shown in the subfigures (e)-(f), the signal-anticipating control has less restrictive convergence condition than the signal-taking control, which is consistent with the theoretical analysis in Section \ref{sect:con-comp}. 
	
	
	\begin{figure}[t!]
		\begin{subfigure}{0.48\columnwidth}
			\caption{Signal-taking: $\alpha_{i} = 9$.} \vspace{1pt}\vspace{-6pt}
			\label{fig:taking9}
			\includegraphics[width=\linewidth]{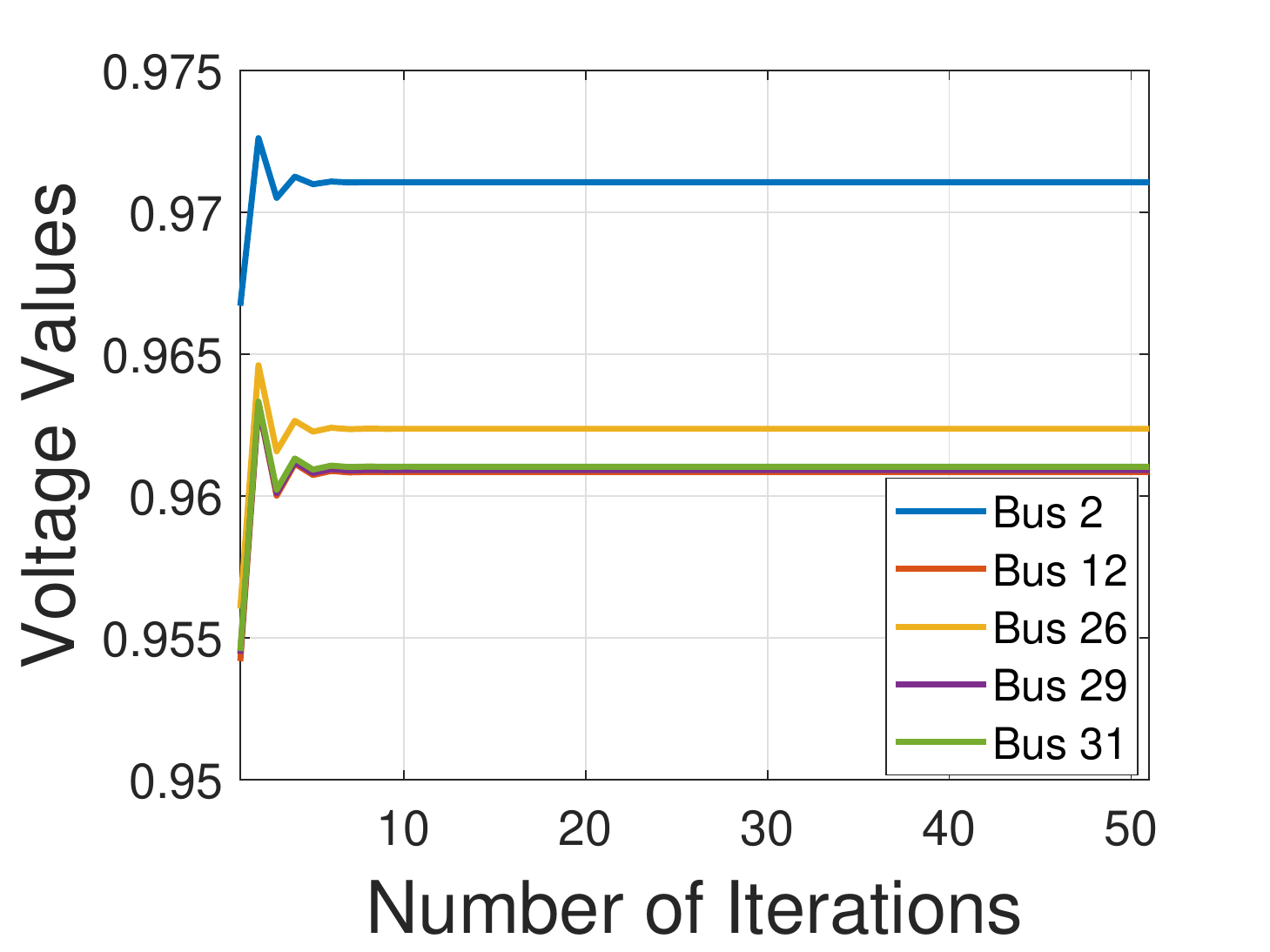}
		\end{subfigure}
		\vspace{13pt}
		\hspace*{\fill}
		\begin{subfigure}{0.48\columnwidth}
			\caption{Signal-anticipating: $\alpha_{i} = 9$.} \label{fig:anti9}\vspace{-6pt}
			\includegraphics[width=\linewidth]{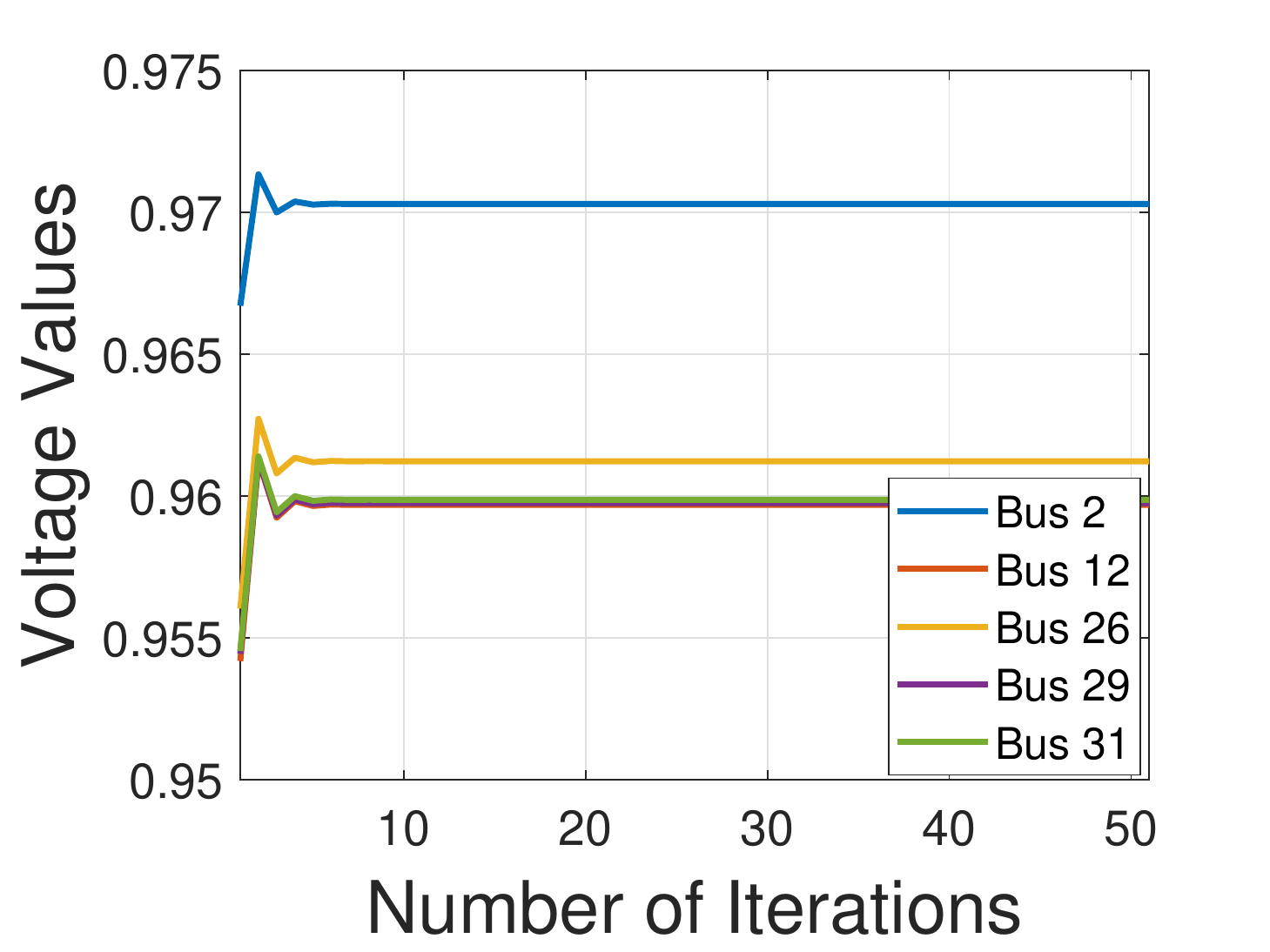}
		\end{subfigure}
		\hspace*{\fill}
		\begin{subfigure}{0.48\columnwidth}
			\caption{Signal-taking: $\alpha_{i}= 18$.} 
			\label{fig:taking18}\vspace{-6pt}
			\includegraphics[width=\linewidth]{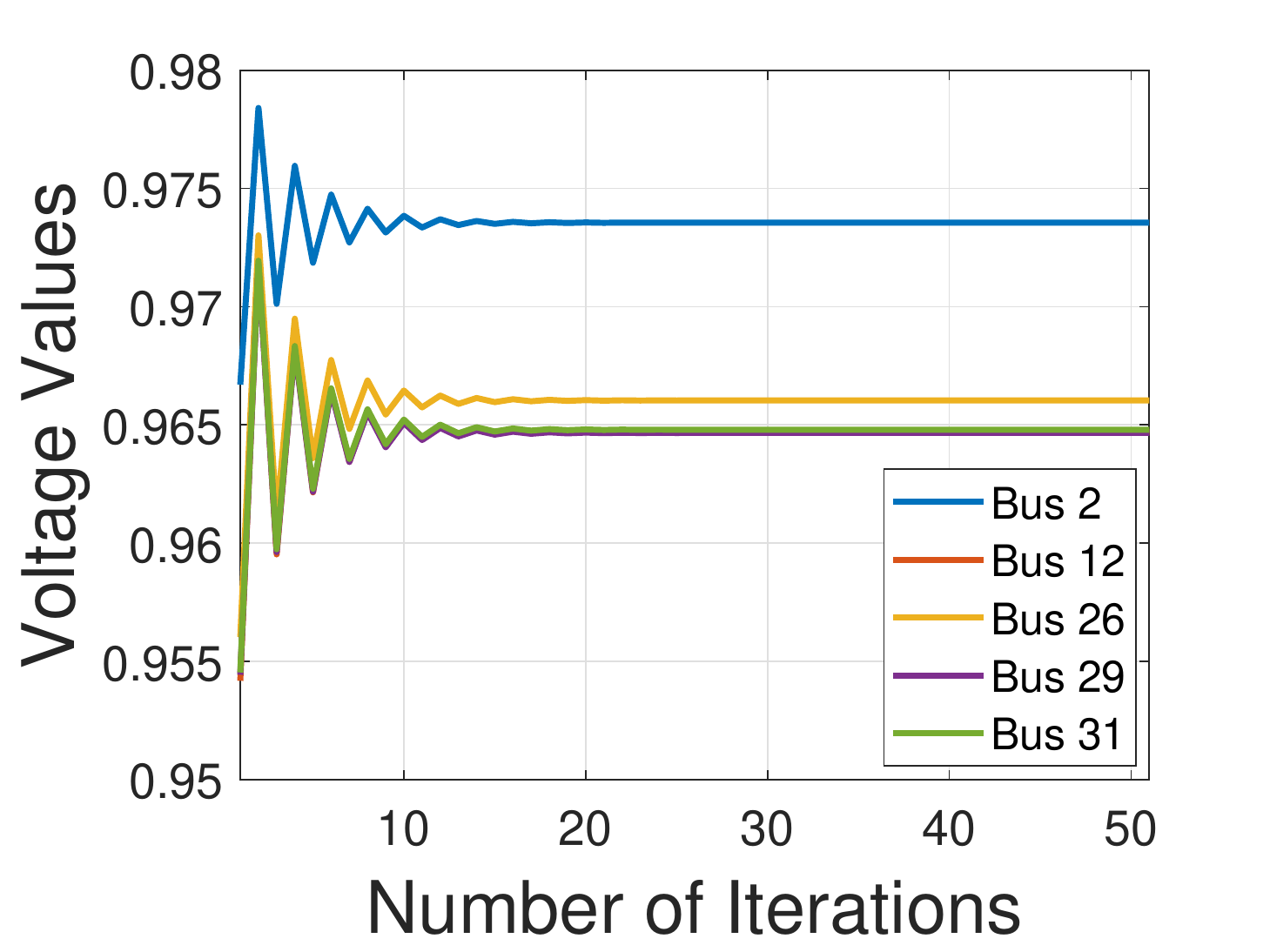}
		\end{subfigure}
	\vspace{13pt}
		\hspace*{\fill}
		\begin{subfigure}{0.48\columnwidth}
			\caption{Signal-anticipating: $\alpha_{i} = 18$.}
			\label{fig:anti18}\vspace{-6pt}
			\includegraphics[width=\linewidth]{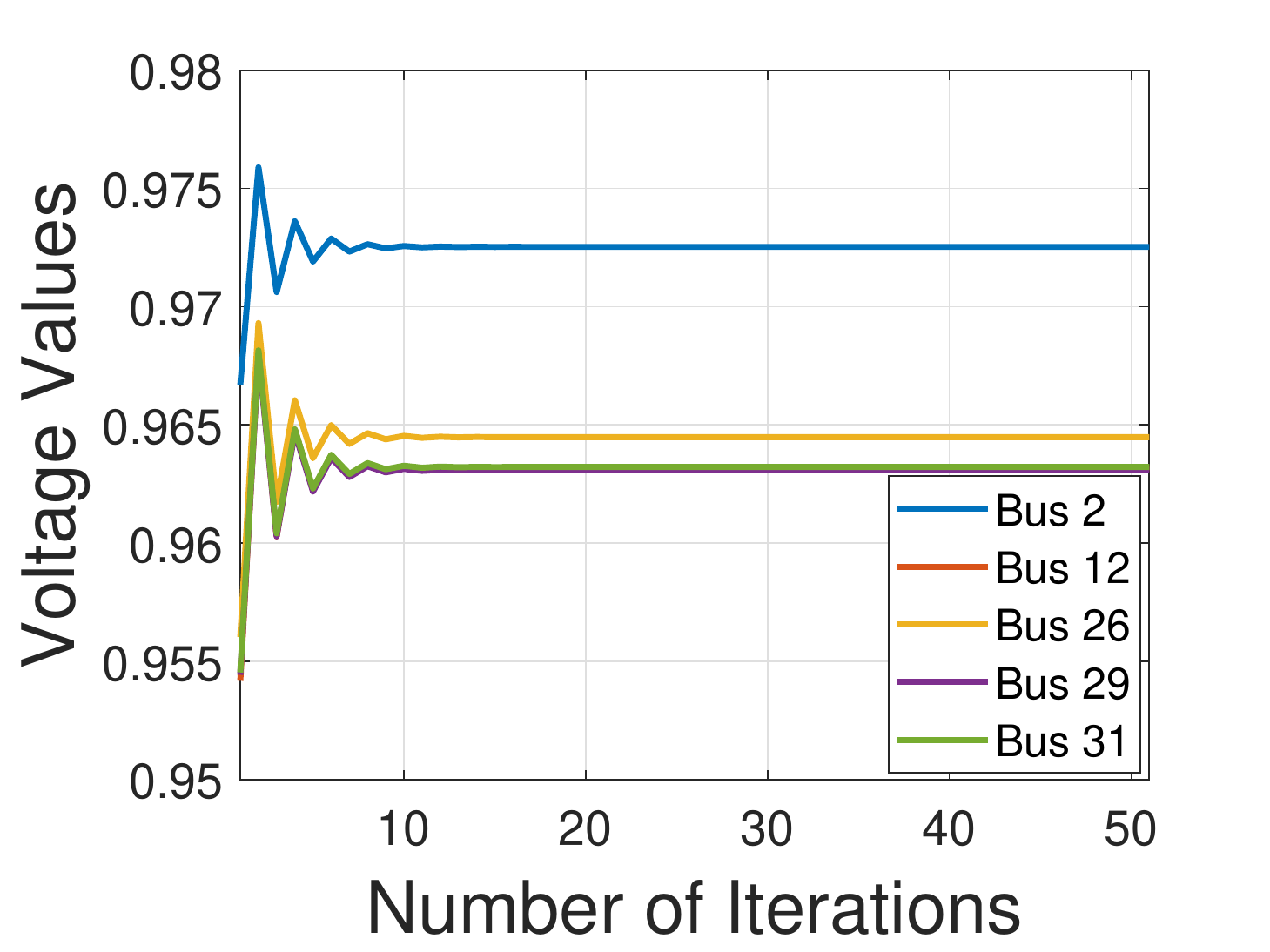}
		\end{subfigure}	
		\hspace*{\fill}
		\begin{subfigure}{0.48\columnwidth}
			\caption{Signal-taking: $\alpha_{i} = 27$.} \label{fig:taking27}\vspace{-6pt}
			\includegraphics[width=\linewidth]{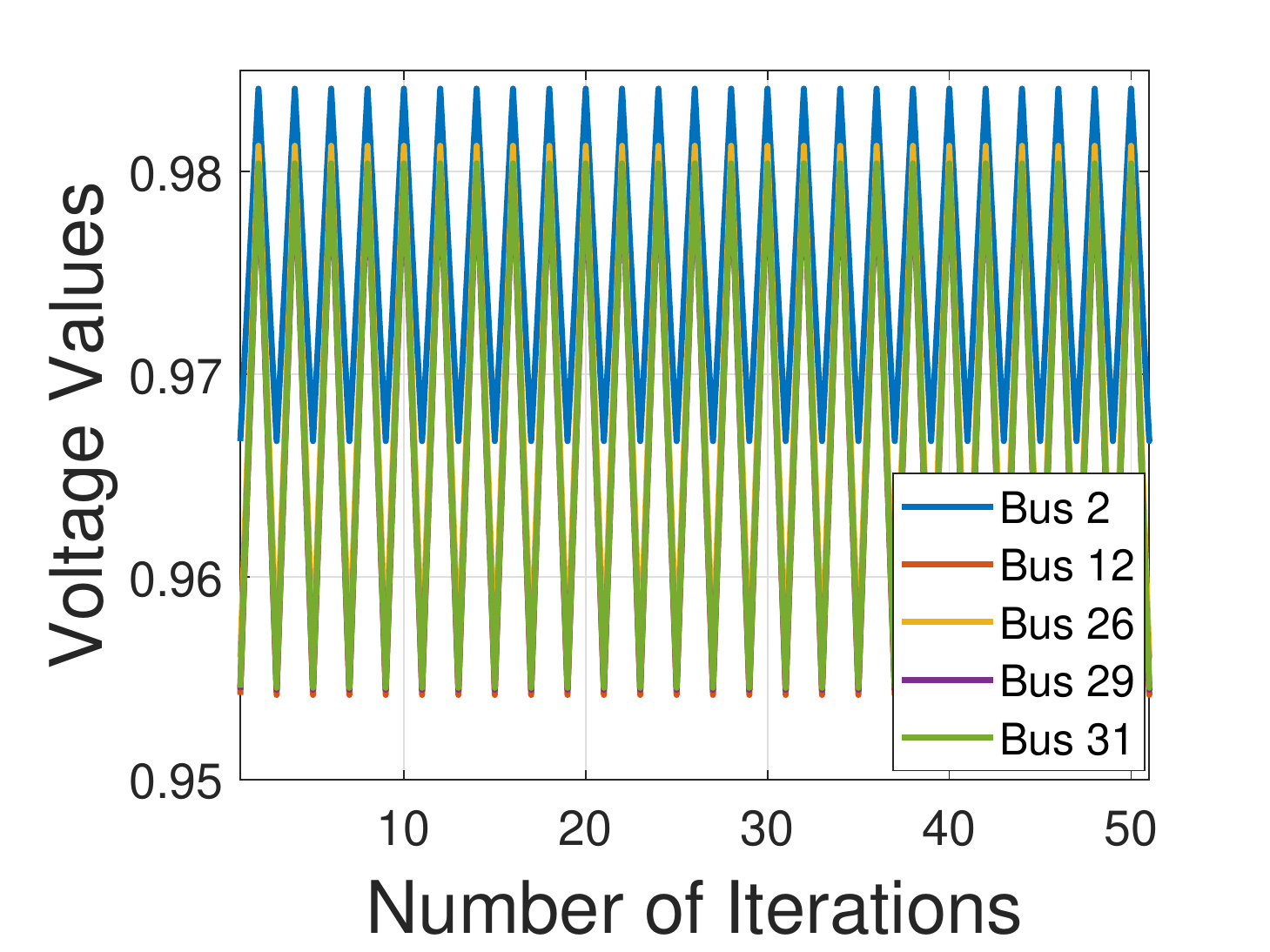}
		\end{subfigure}
		\hspace*{\fill}
		\begin{subfigure}{0.48\columnwidth}
			\caption{Signal-anticipating: $\alpha_{i}= 27$.} \label{fig:anti27}\vspace{-6pt}
			\includegraphics[width=\linewidth]{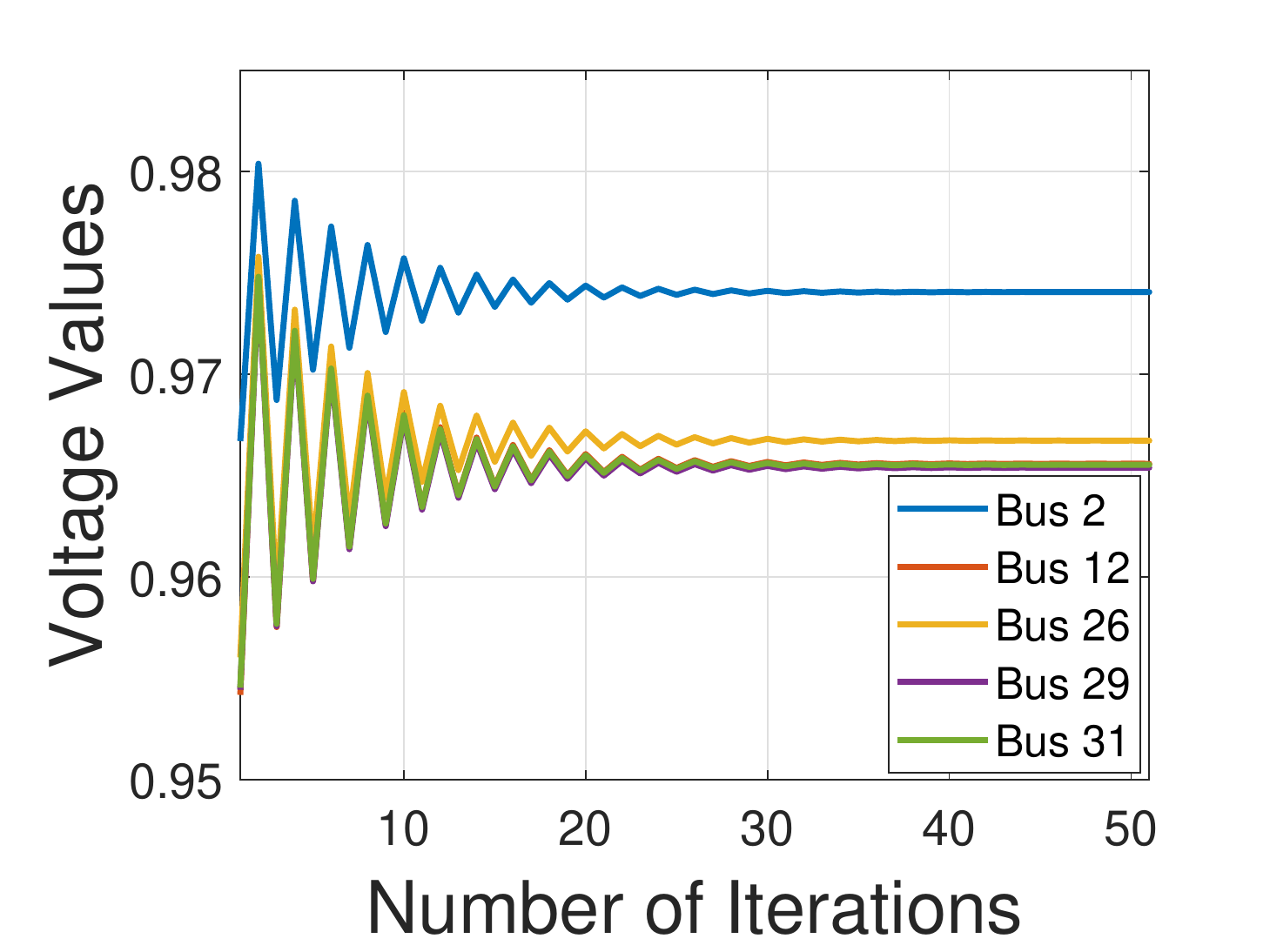}
		\end{subfigure}	
		\caption{Convergence comparison between the signal-anticipating and signal-taking voltage controls  with a deadband $\delta_i=0.02 p.u.$ and PV capacity constraints.} \label{fig:exper}
	\end{figure}

	\begin{figure}[t!]
	\begin{subfigure}{0.48\columnwidth}
		\caption{Signal-taking: $\alpha_{i} = 9$.} \vspace{1pt}\vspace{-6pt}
		\label{fig:taking9}
		\includegraphics[width=\linewidth]{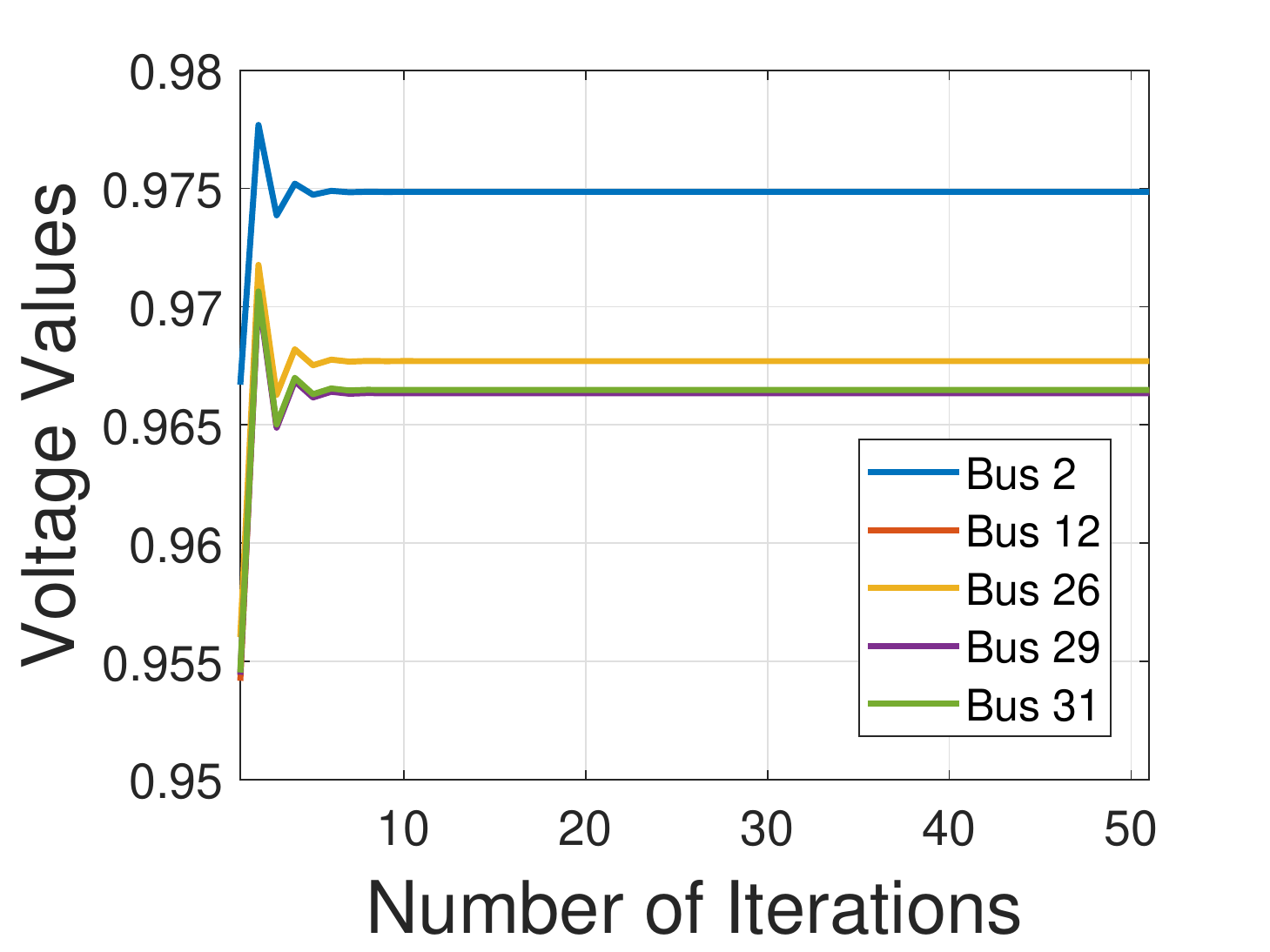}
	\end{subfigure}
	\vspace{13pt}
	\hspace*{\fill}
	\begin{subfigure}{0.48\columnwidth}
		\caption{Signal-anticipating: $\alpha_{i} = 9$.} \label{fig:anti9}\vspace{-6pt}
		\includegraphics[width=\linewidth]{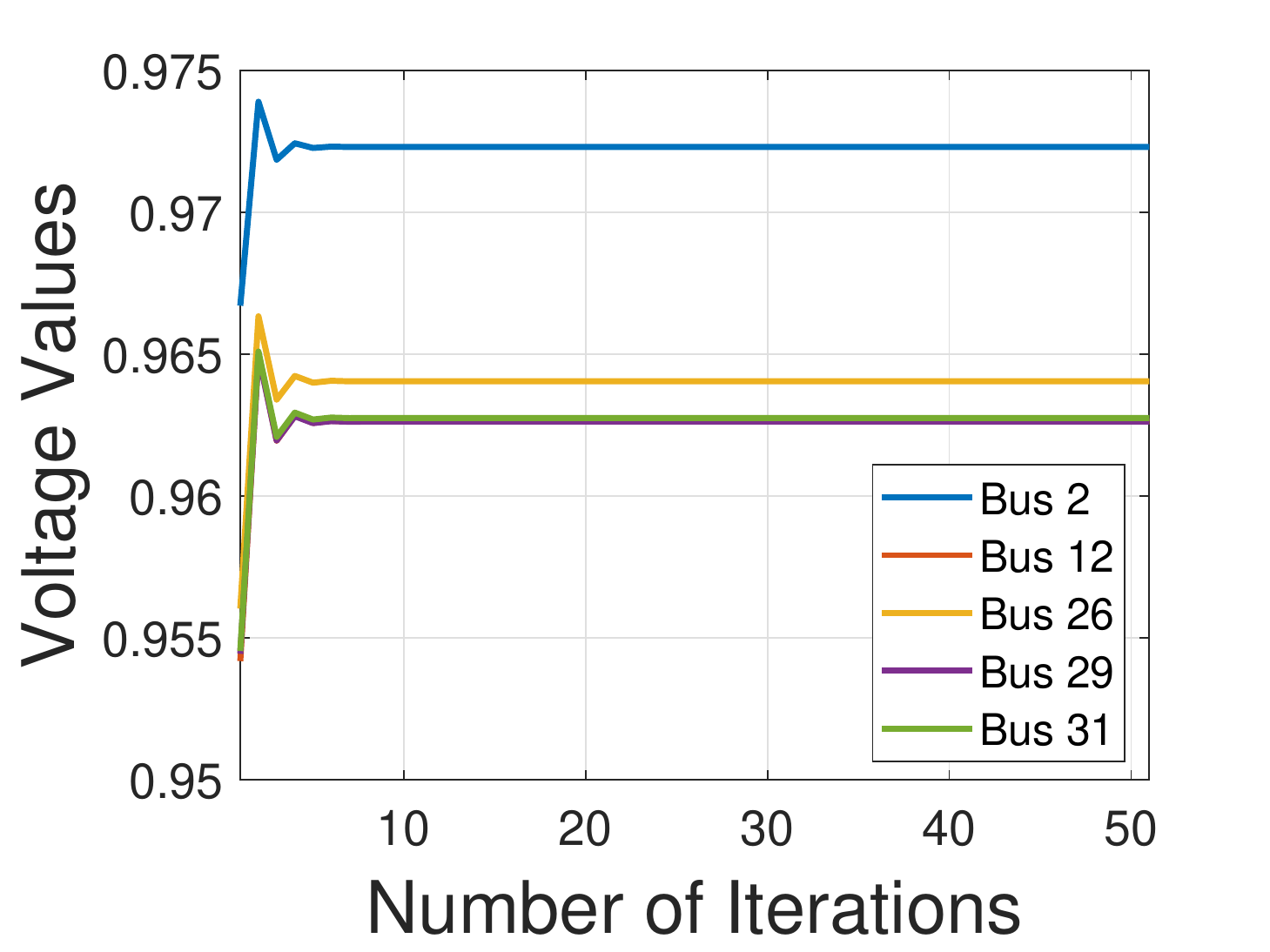}
	\end{subfigure}
	\hspace*{\fill}
	\begin{subfigure}{0.48\columnwidth}
		\caption{Signal-taking: $\alpha_{i}= 18$.} 
		\label{fig:taking18}\vspace{-6pt}
		\includegraphics[width=\linewidth]{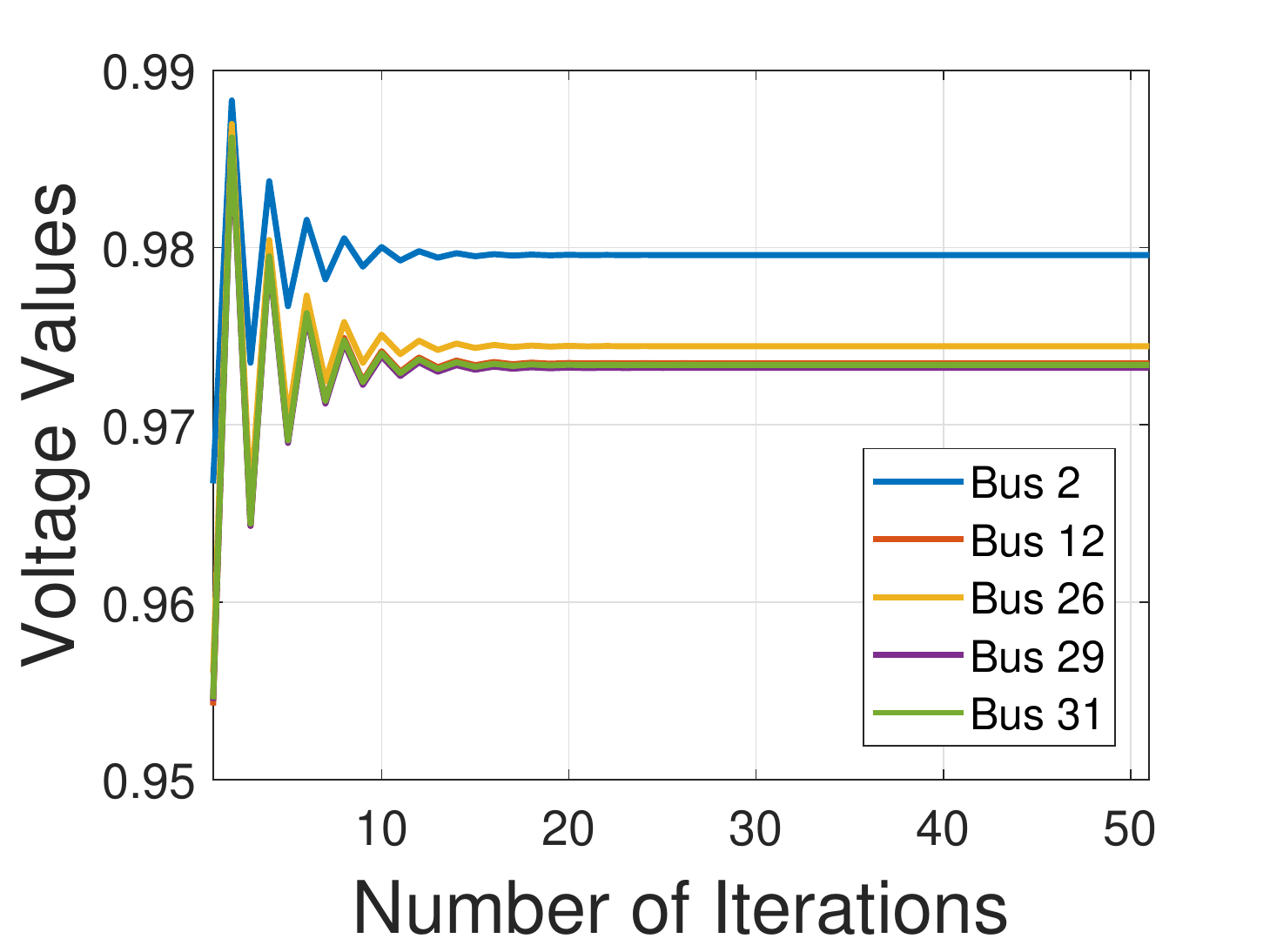}
	\end{subfigure}
	\vspace{13pt}
	\hspace*{\fill}
	\begin{subfigure}{0.48\columnwidth}
		\caption{Signal-anticipating: $\alpha_{i} = 18$.}
		\label{fig:anti18}\vspace{-6pt}
		\includegraphics[width=\linewidth]{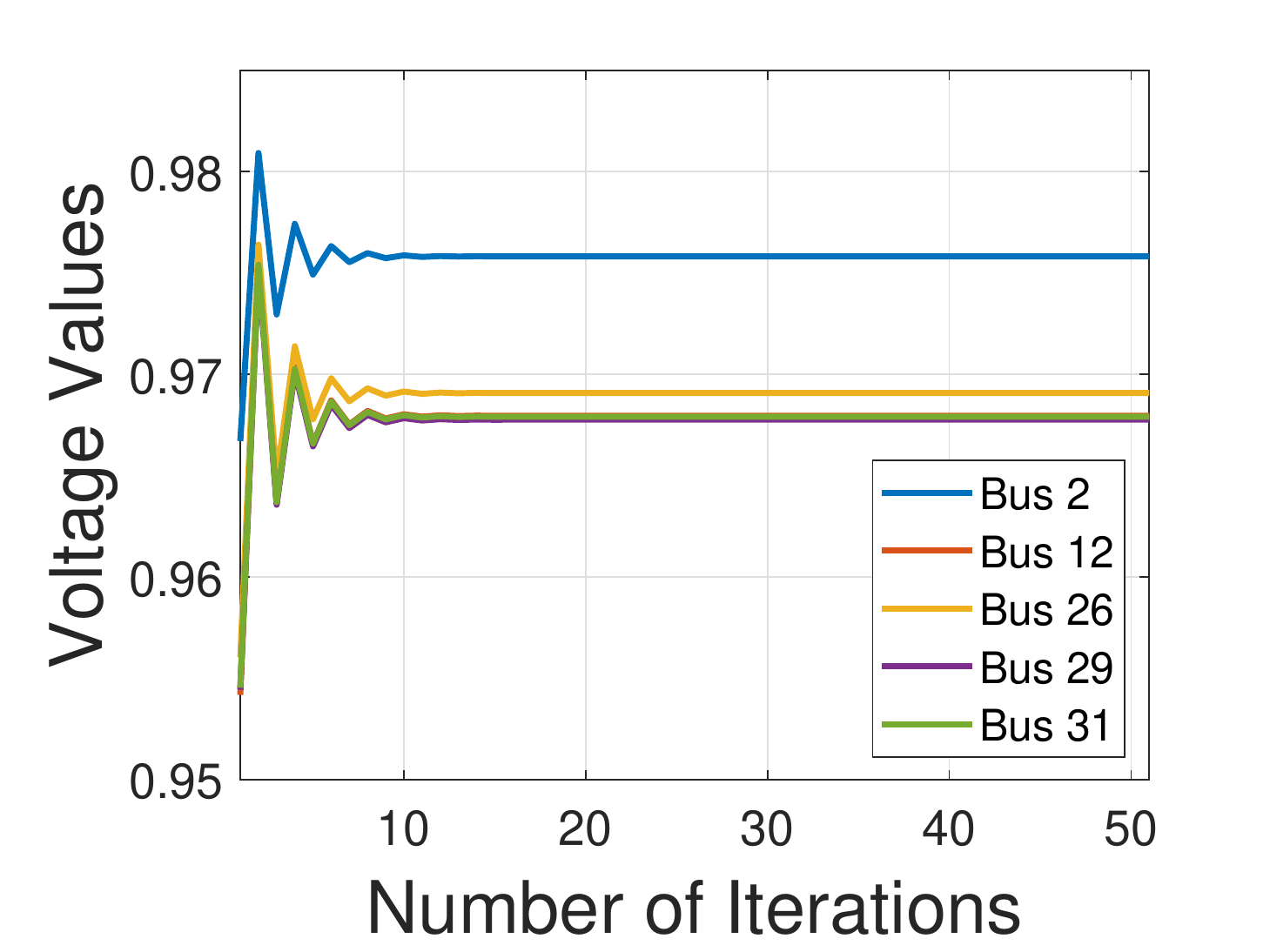}
	\end{subfigure}	
	\hspace*{\fill}
	\begin{subfigure}{0.48\columnwidth}
		\caption{Signal-taking: $\alpha_{i} = 27$.} \label{fig:taking27}\vspace{-6pt}
		\includegraphics[width=\linewidth]{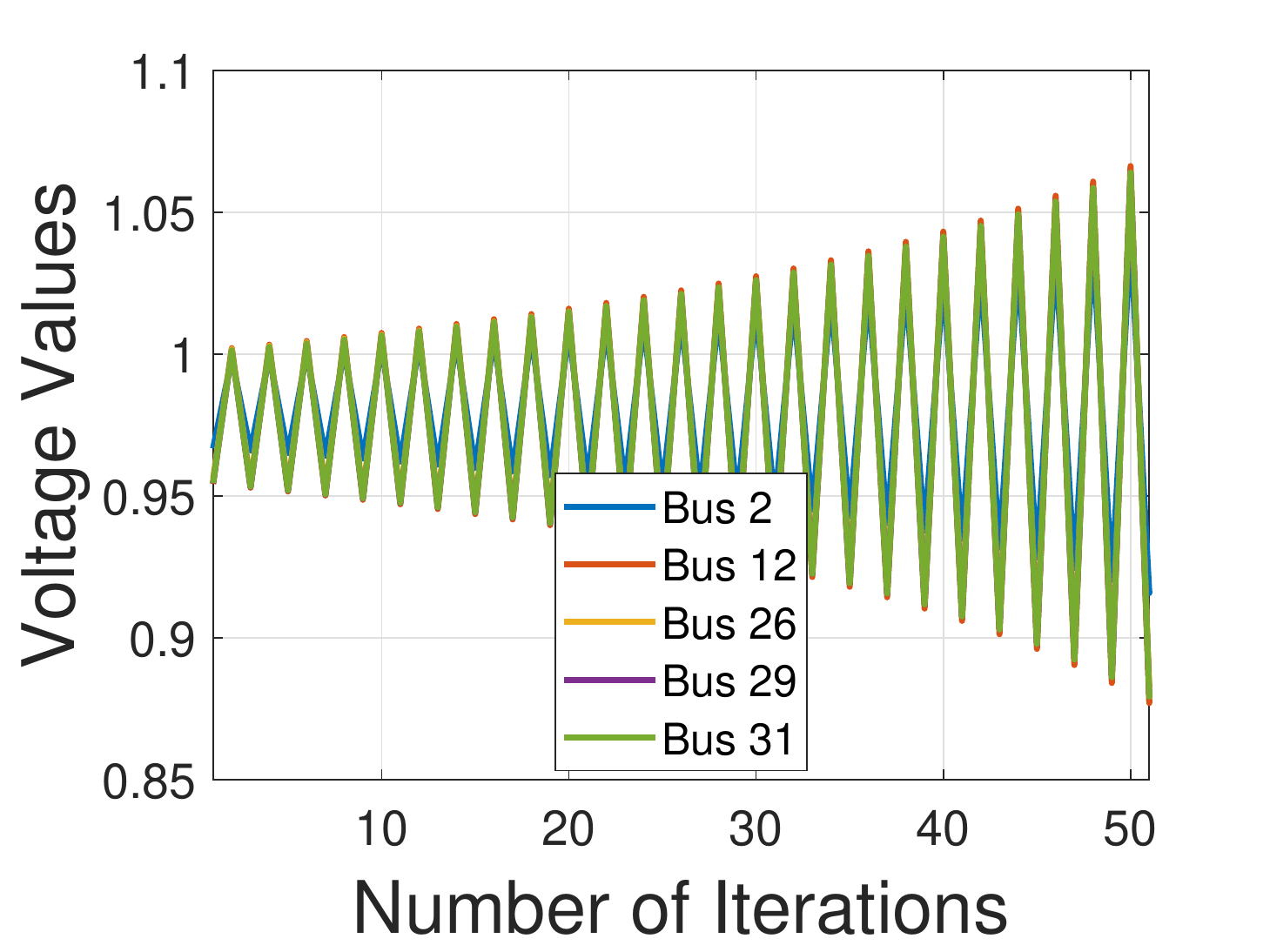}
	\end{subfigure}
	\hspace*{\fill}
	\begin{subfigure}{0.48\columnwidth}
		\caption{Signal-anticipating: $\alpha_{i}= 27$.} \label{fig:anti27}\vspace{-6pt}
		\includegraphics[width=\linewidth]{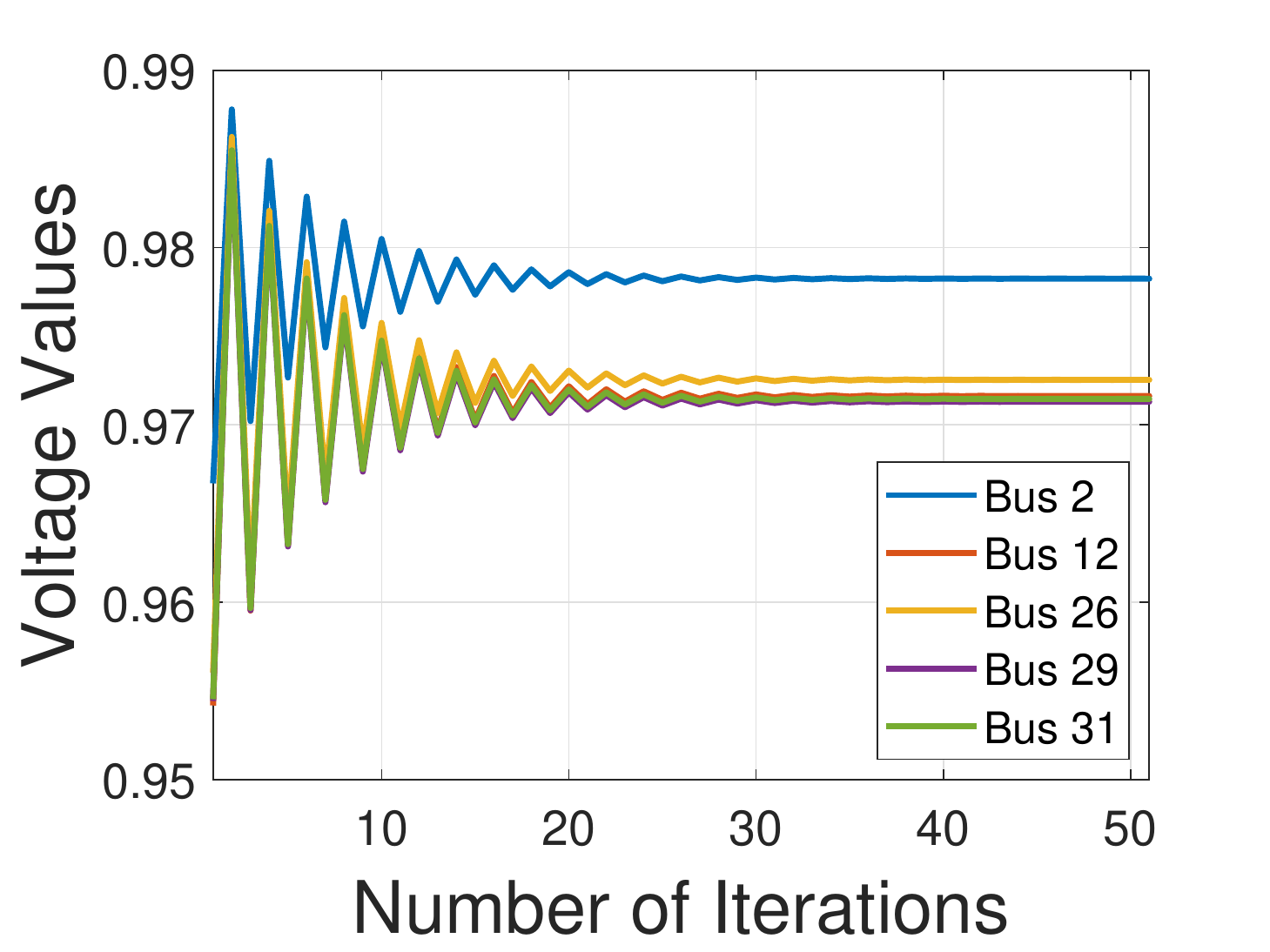}
	\end{subfigure}	
	\caption{Convergence comparison between the signal-anticipating and signal-taking voltage controls  without deadband or PV capacity constraints.} \label{fig:expern}
\end{figure}

	\begin{table*}[t]
		\centering
		
		\begin{tabular}{|c|c|c|c|c|c|c|c|c|c|c|c|c|c|c|c|c|c|}
			\hline
			\multicolumn{18}{|c|}{Network Data}\\
			\hline
			\multicolumn{4}{|c|}{Line Data}& \multicolumn{4}{|c|}{Line Data}& \multicolumn{4}{|c|}{Line Data}& \multicolumn{2}{|c|}{Load Data}& \multicolumn{2}{|c|}{Load Data}&\multicolumn{2}{|c|}{\!\!PV Generators\!\!}\\
			\hline
			From&To&R&X&From&To& R& X& From& To& R& X& Bus& Peak & Bus& Peak & Bus&\!\!Capacity\!\!\\
			Bus.&Bus.&$(\Omega)$& $(\Omega)$ & Bus. & Bus. & $(\Omega)$ & $(\Omega)$ & Bus.& Bus.& $(\Omega)$ & $(\Omega)$ & No.&  MVA& No.& MVA& No.& MW\\
			\hline
			1	&	2	&	0.259	&	0.808	&	8	&	34	&	0.244	&	0.046 	&	18	&	19	&	0.198	&	 0.046	&	11	&	 0.67	&	28	&	 0.27 	&				 &		 \\

			2	&	3	&	0.031	&	0.092	&	8	&	36	&	0.107	 &	0.031 	&	22	&	26	&	0.046	&	 0.015	&	12	&	0.45		 &	29	&	0.2 &			 2	&	 1\\

			3	&	4	&	0.046	&	0.092	&	8	&	30	&	0.076	 &	0.015 	&	22	&	23	&	0.107	&	 0.031	&	13	&	0.89	 &	31	&	0.27	  &			26	&	 2	 \\

			3	&	13	&	0.092	&	0.031	&	8	&	9	&	0.031	 &	0.031 	&	23	&	24	&	0.107	&	 0.031	&	15	&	0.07	 &	33	&	0.45	 	 &			29	&	 1.8 	 \\

			3	&	14	&	0.214	&	0.046	&	9	&	10	&	0.015	 &	0.015	&	24	&	25	&	0.061	&	 0.015	&	16	&	0.67	 &	34	&	1.34  &			31	&	 2.5 	 \\

			4	&	17	&	0.336	&	0.061	&	 9	&	37	&	0.153	 &	0.046 	&	27	&	28	&	0.046	&	 0.015	&	18	&	0.45	 &	35	&	0.13	  &			12	&	 3 	 \\

			4	&	5	&	0.107	&	0.183	&	10	&	11	&	0.107	 &	0.076 	&	28	&	29	&	0.031	&	0		&	19	&	1.23	 &	36	&	0.67	  &			&		 \\

			5	&	21	&	0.061	&	0.015	&	10	&	41	&	0.229	 &	0.122 	&	30	&	31	&	0.076	&	 0.015  &	20	&	0.45	 &	37	&	0.13	 	&			 &			 \\

			5	&	6	&	0.015	&	0.031	&	11	&	42	&	0.031	 &	0.015 	&	30	&	32	&	0.076	&	 0.046	&	21	&	0.2 &	39	&	0.45 	&			  &	 \\

			6	&	22	&	0.168	&	0.061	&	11	&	12	&	0.076	 &	0.046 	&	30	&	33	&	0.107	&	 0.015	&	23	&	0.13		 &	40	&	0.2 	&		 &	\\

			6	&	7	&	0.031	&	0.046	&	14	&	16	&	0.046	 &	0.015 	&	37	&	38	&	0.061	&	0.015	&	24	&	0.13	 &	 41	&	0.45		&		&		 \\		\cline{15-18}

			7	&	27	&	0.076	&	0.015	&	14	&	15	&	0.107	 &	0.015	&	38	&	39	&	0.061	&	0.015	&	 25	&	0.2 &	 \multicolumn{4}{c|}{$V_{base}$ = 12.35 KV}		 	\\

			7	&	8	&	0.015	&	0.015	&	17	&	18	&	0.122	 &	0.092	&	38	&	40	&	0.061	&	 0.015	&	 26	&	0.07 &	 \multicolumn{4}{c|}{$S_{base}$ = 1000 KVA} 	 	\\

			8	&	35	&	0.046	&	0.015	&	17	&	20	&	0.214	 &	0.046	&	 	&	 	&	 	&	 	&	27	&	0.13		 &	 \multicolumn{4}{c|}{$Z_{base}$ = 152.52 $\Omega$}   	 	 \\	
			
			\hline
		\end{tabular}
		\caption{Network Parameters of the SCE Circuit: Line impedances, peak spot load KVA, Capacitors and PV generation's nameplate ratings. }
		\label{data}
	\end{table*}

}
	\section{Conclusion}
	

	We have considered the signal-anticipating behavior in local voltage control for distribution systems, and shown that the signal-anticipating voltage control is the best response algorithm of a voltage control game. We characterize the Nash equilibrium of the voltage control game and establish its asymptotic global stability under the signal-anticipating voltage control. We have further introduced the notion of Price of Signal-Anticipation (PoSA) to characterize the impact of signal-anticipating control, and characterized analytically and numerically how the PoSA scales with the size, topology, and heterogeneity of the network. 	Our results show that the PoSA is upper bounded by a constant and the average PoSA per node will go to zero as the size of the network increases. This is desirable as it means that the PoSA will not be arbitrarily large, no matter what the size of the network is, and no mechanism is needed to mitigate the signal-anticipating behavior. 

	\vspace{3mm}
 	\bibliographystyle{IEEEtran}
	\bibliography{posa,power}

{	
\appendix[Properties of Reactance Matrix $X$] 


One interesting and important property of $X$ is that its inverse $X^{-1}$ has an analytical form which is strongly related to the Laplacian matrix of the inverse tree defined next.
\begin{definition}[\textbf{Inverse Tree}]
	Given the tree graph $\mathcal{T}=\{\hN \cup\{0\}, \hL\}$ with node $0$ (root node) being of degree 1, the inverse tree $\mathcal{T}'$ is defined as follows:
	\begin{itemize}
		\item[a.] $\mathcal{T}'$ and $\mathcal{T}$ have the same topology, i.e., the same sets of vertices and lines. 
		\item[b.] For each line $(i,j) \in \hL$ of $\mathcal{T}'$, its link weight is $x_{ij}' = \frac{1}{x_{ij}}.$
	\end{itemize} 
\end{definition}
See Fig.~\ref{fig:tree1} and Fig.~\ref{fig:tree2} for an illustrative example for inverse tree. 
\begin{figure}[!ht]
	\centering
	\begin{minipage}{0.4\linewidth}
		\begin{tikzpicture}[level distance=1.cm,
		level 1/.style={sibling distance=1cm},
		level 2/.style={sibling distance=1cm}]
		\tikzstyle{every node}=[circle,draw]
		
		\node [rectangle]{0 (Root)}
		child {
			node {1} 
			child {
				node {2}  	
				child {
					node {3} 
					edge from parent
					node[right,draw=none]{$c$}
				}
				edge from parent 
				node [left,draw=none]{$b$}
			}
			child[missing]
			child {
				node {4} 
				edge from parent
				node[right,draw=none]{$d$}
			}
			edge from parent 
			node[left,draw=none]{$a$}
		};
		\end{tikzpicture} 
		\caption{The original \\ tree $\mathcal{T}$.}
		\label{fig:tree1}
	\end{minipage}~
	\begin{minipage}{0.36\linewidth}
		\begin{tikzpicture}[level distance=1cm,
		level 1/.style={sibling distance=1cm},
		level 2/.style={sibling distance=1cm}]
		\tikzstyle{every node}=[circle,draw]
		
		\node [rectangle] {0 (Root)}
		child {
			node {1} 
			child {
				node {2}  	
				child {
					node {3} 
					edge from parent
					node[right,draw=none]{${1}/{c}$}
				}
				edge from parent 
				node [left,draw=none]{${1}/{b}$}
			}
			child[missing]
			child {
				node {4} 
				edge from parent
				node[right,draw=none]{${1}/{d}$}
			}
			edge from parent 
			node[left,draw=none]{${1}/{a}$}
		};
		\end{tikzpicture} 
		\caption{The inverse \\ tree $\mathcal{T}'$.}
		\label{fig:tree2}
	\end{minipage}
\end{figure}

Denote by $\mathbb{L} \in \mathbb{R}^{n \times n}$ the weighted Laplacian matrix for inverse tree $\mathcal{T}'$ excluding the root node, defined as follows
\begin{eqnarray}
\mathbb{L}_{ij}=\left\{ \begin{array}{ll}
-{1}/{x_{ij}}, & (i,j) \in \hL,i\neq j\\
\sum_{(i,k) \in \hL} {1}/{x_{ik}}, &i = j\\
0, &\text{otherwise}
\end{array}\right..\nonumber
\end{eqnarray}
Index by $1$ the direct child node of node $0$ and denote by $a$ the weight of link $(0, 1)$. We have the following result that gives an explicit form of $X^{-1}$.
\begin{lemma}\label{Laplacian}
	For any tree $\mathcal{T}$, the inverse $X^{-1}$ of the reactance matrix $X$ has the following explicit form:
	\begin{equation}\label{equ:inverse}
	\begin{aligned}
	X^{-1} = \mathbb{L} + \begin{bmatrix}
	{1}/{a}&0&\cdots\\
	0& 0 & \cdots \\
	\vdots&\vdots&\ddots\\
	\end{bmatrix}. 
	\end{aligned}
	\end{equation}
\end{lemma}
\begin{proof}
	Before proving the result, notice that the reactance matrix, its inverse, and the weighted Laplacian matrix are for the tree and the inverse tree excluding the root node (node $0$), even though they contain information about the root node. 
	Accordingly, denote by $\mathbb{N}_i$, $\mathbb{C}_i$ and $h_i$ the set of neighbor nodes,  the set of child nodes and the parent node of node $i\in\cN$ in the tree $\mathcal{T}$ (and the inverse tree $\mathcal{T}^{'}$) excluding the root node. 
	Notice that 
	\begin{eqnarray}
	(X^{-1})_{ij}=\begin{cases}
	{1}/{a}+\sum_{k\in \mathbb{C}_1}{1}/({X_{kk}-X_{11}}), & \text{if }i=j=1, \\
	{1}/({X_{ii}-X_{h_i h_i}}) +\sum_{k\in \mathbb{C}_i}{1}/({X_{kk}-X_{ii}}),& \text{if }i=j\neq 1,\\
	-{1}/{|X_{ii} - X_{jj}|}, & \text{if }j \in \mathbb{N}_i,\\
	0,&\text{otherwise.}
	\end{cases}\nonumber
	\end{eqnarray}
	
	
	To prove \eqref{equ:inverse}, it is equivalent to show  $(XX^{-1})_{ij} = \delta_{ij}$, where $\delta_{ij}$ is the Kronecker delta. We first analyze $(XX^{-1})_{ii}$. When $i \neq 1$, 
	\begin{eqnarray}
	(XX^{-1})_{ii} &=& \sum_{k}X_{ik}(X^{-1})_{ki} \nonumber\\
	&=& X_{ii}(X^{-1})_{ii} + \sum_{k \in \mathbb{N}_i}X_{ik}(X^{-1})_{ki} \nonumber\\
	&=& X_{ii}(X^{-1})_{ii} + \frac{X_{h_i h_i}}{X_{h_i h_i}-X_{ii}}+ \sum_{k \in \mathbb{C}_i}\frac{X_{ii}}{X_{ii}-X_{kk}} \nonumber\\
	&=& 1. \nonumber
	\end{eqnarray}
	When $i=1$, 
	\begin{equation*}
	\begin{aligned}
	(XX^{-1})_{11} = X_{11}(X^{-1})_{11} + \sum_{k \in \mathbb{C}_1}\frac{X_{11}}{X_{11}-X_{kk}} = a\frac{1}{a} =1.
	\end{aligned}
	\end{equation*}   
	
	Similarly, for $(XX^{-1})_{ij}$ with $i\neq j$, we have
	\begin{equation*}
	\begin{aligned}
	(XX^{-1})_{ij} 
	= \frac{X_{ij}-X_{ih_j}}{X_{jj}-X_{h_j h_j}} + \sum_{k \in \mathbb{C}_j}\frac{X_{ij}-X_{ik}}{X_{kk}-X_{jj}}. 
	\end{aligned}
	\end{equation*}
	Consider the overlap $\hL_i\cap\hL_j$ between the paths from the root node to nodes $i$ and $j$. We have two cases:
	\begin{itemize}
		\item If $\hL_i\cap\hL_j \subseteq \hL_{h_j}$, then $X_{ij}=X_{ik}$ for $k\in \mathbb{N}_j$ and thus $(XX^{-1})_{ij}=0$. 
		\item Otherwise, $\hL_{j} \subset \hL_i$ and there exits a node $c_j \in \mathbb{C}_j$ such that  $\hL_{c_j} \subseteq \hL_i$. We then have $X_{ij}=X_{jj}$, $X_{ih_j}=X_{h_j h_j}$, $X_{ic_j}=X_{c_j c_j}$ and $X_{ik}=X_{jj}$ for $k\in  \mathbb{C}_j \backslash \{c_j\}$. Applying these relations, we have $(XX^{-1})_{ij}=0$.   
	\end{itemize}	 
\end{proof}

As an illustrative example, we calculate $X^{-1}$ based on Fig.~\ref{fig:tree1} and \ref{fig:tree2}.
The reactance matrix $X$ for the original tree in Fig.~\ref{fig:tree1} is
\begin{equation*}
\begin{aligned}
X = \begin{bmatrix}
a&a&a&a\\
a&a+b&a+b&a\\
a&a+b&a+b+c&a\\
a&a&a&a+d 
\end{bmatrix}.
\end{aligned}
\end{equation*}
Based on the corresponding inverse tree in Fig.~\ref{fig:tree2}, the weighted Laplacian matrix $\mathbb{L}$ of $\mathcal{T}'$ excluding node $0$ is 
\begin{equation*}
\begin{aligned}
\mathbb{L} = \begin{bmatrix}
(b+d)/bd&-1/b&0&-1/d\\
-1/b&(b+c)/bc&-1/c&0\\
0&-1/c&1/c&0\\
-1/d&0&0&1/d 
\end{bmatrix}.
\end{aligned}
\end{equation*}
By Lemma~\ref{Laplacian}, we have
\begin{equation*}
\begin{aligned}
X^{-1} = \begin{bmatrix}
(b\!+\!d)/bd\!+\!1/a&-1/b&0&-1/d\\
-1/b&(b+c)/bc&-1/c&0\\
0&-1/c&1/c&0\\
-1/d&0&0&1/d 
\end{bmatrix}.
\end{aligned}
\end{equation*}

As can be seen from Lemma \ref{Laplacian} and the above example, the inverse matrix $X^{-1}$ reveals the topology information of power network (through the Laplacian matrix $\mathbb{L}$). 
}

\end{document}